\documentclass[onefignum,onetabnum]{siamart190516}

\usepackage{lipsum}
\usepackage{amsfonts}
\usepackage{graphicx}
\usepackage{epstopdf}
\usepackage{algorithmic}
\usepackage{enumitem}
\usepackage{color}
\usepackage{animate}
\usepackage{amsmath}
\usepackage{amssymb}
\usepackage{graphics}
\usepackage{graphicx}
\usepackage{footnote}
\usepackage{textcomp}
\usepackage{mathrsfs}
\usepackage{epstopdf}
\usepackage{array}
\usepackage[maxfloats=99]{morefloats}
\usepackage{url}
\usepackage{cases}
\usepackage{mathscinet}
\usepackage[normalem]{ulem}
\usepackage{algorithmic, algorithm}
\usepackage{color}
\usepackage{cite}
\usepackage{bm}
\usepackage{stmaryrd}

\usepackage{subcaption}

\usepackage{xparse}

\NewDocumentCommand{\dgal}{sO{}m}{%
  \IfBooleanTF{#1}
    {\dgalext{#3}}
    {\dgalx[#2]{#3}}%
}

\NewDocumentCommand{\dgalext}{m}{%
  \sbox0{%
    \mathsurround=0pt % just for safety
    $\left\{\vphantom{#1}\right.\kern-\nulldelimiterspace$%
  }%
  \sbox2{\{}%
  \ifdim\ht0=\ht2
    \{\kern-.45\wd2 \{#1\}\kern-.45\wd2 \}%
  \else
    \left\{\kern-.5\wd0\left\{#1\right\}\kern-.5\wd0\right\}%
  \fi
}

\NewDocumentCommand{\dgalx}{om}{%
  \sbox0{\mathsurround=0pt$#1\{$}%
  \sbox2{\{}%
  \ifdim\ht0=\ht2
    \{\kern-.45\wd2 \{#2\}\kern-.45\wd2 \}%
  \else
    \mathopen{#1\{\kern-.5\wd0 #1\{}
    #2
    \mathclose{#1\}\kern-.5\wd0 #1\}}
  \fi
}

\DeclareFontEncoding{FMS}{}{}
\DeclareFontSubstitution{FMS}{futm}{m}{n}
\DeclareFontEncoding{FMX}{}{}
\DeclareFontSubstitution{FMX}{futm}{m}{n}
\DeclareSymbolFont{fouriersymbols}{FMS}{futm}{m}{n}
\DeclareSymbolFont{fourierlargesymbols}{FMX}{futm}{m}{n}
\DeclareMathDelimiter{\VERT}{\mathord}{fouriersymbols}{152}{fourierlargesymbols}{147}
\ifpdf
  \DeclareGraphicsExtensions{.eps,.pdf,.png,.jpg}
\else
  \DeclareGraphicsExtensions{.eps}
\fi

\newsiamremark{remark}{Remark}
\newsiamremark{hypothesis}{Hypothesis}
\crefname{hypothesis}{Hypothesis}{Hypotheses}
\newsiamthm{claim}{Claim}

\headers{Unfitted computation of band structure}{H. Guo, X. Yang and Y. Zhu}

\title{Unfitted Nitsche's method for computing band structures in phononic crystals with impurities}
\author{Hailong Guo\thanks{School of Mathematics and Statistics,  The University of Melbourne,  Parkville, VIC 3010, Australia   (hailong.guo@unimelb.edu.au).}
\and %
Xu Yang\thanks{Department of Mathematics, University of California, Santa Barbara, CA, 93106, USA (xuyang@math.ucsb.edu).}
\and%
Yi Zhu\thanks{Yau Mathematical Sciences Center and Department of Mathematical Sciences, Tsinghua University, Beijing, 100084, People's Republic of China (yizhu@mail.tsinghua.edu.cn).}
}

\usepackage{amsopn}

\ifpdf
\hypersetup{
  pdftitle={Unfitted Nitsche's method for computing band structures in phononic crystals with impurities},
  pdfauthor={H. Guo, X. Yang and Y. Zhu}
}
\fi

\begin{document}

\maketitle

% REQUIRED
\begin{abstract}
In this paper, we propose an unfitted Nitsche's method to compute the band structures of phononic crystal with impurities of general geometry. The proposed method does not require the background mesh to fit the interfaces of impurities, and thus avoids the expensive cost of generating body-fitted meshes and simplifies the inclusion of interface conditions in the formulation. The quasi-periodic boundary conditions are handled by the Floquet-Bloch transform, which converts the computation of band structures into an eigenvalue problem with periodic boundary conditions. More importantly, we show the well-posedness of the proposed method using a delicate argument based on the trace inequality, and further prove the convergence by the Babu\v{s}ka-Osborn theory. We achieve the optimal convergence rate at the presence of the impurities of general geometry. We confirm the theoretical results by two numerical examples, and show the capability of the proposed methods for computing the band structures without fitting the interfaces of impurities.

%In this paper,  we propose a new unfitted finite element method to compute the band structure of phononic crystal.  The proposed method is seamlessly infused with the Bloch-Floquet theory and uses the computational meshes independent of the location of the material interface. We show that the discrete problem is well-defined and the discrete eigenvalues optimally converge to the exact eigenvalues even though the existence of heterogeneous structure. Some benchmark numerical example to validate the theoretical results and demonstrate the efficient in the computation of the band structure.

\end{abstract}

% REQUIRED

% REQUIRED

\section{Introduction}
Phononic crystals are synthetic materials with periodic structure. Similar to photonic crystals, they present band-gap structures related to topological properties, which prevent elastic waves propagating in certain frequencies. This leads to a series of important applications such as ultrasound imaging and wireless communications.  In the past several years, the blooming of topological phenomena in phononic materials is taking the investigation on phonoic crystals to a new height. One of the key problems is to obtain the band structure of bulk phononic crystals. In literature, Economou and Sigalas \cite{EcSi1994} experimentally observed the band-gap in phononic crystals.  Ammari {\it et al.} \cite{AmKL2009} mathematically proved the existence of band-gap in the high-contrast phononic crystal using the asymptotic expansion and the generalized  Rouch\'e's theorem.

%({\color{red}need more background about phononic crystal and its application})

In general, phononic crystals with large band-gap is preferred due to the wide range of applications. One of the most influential accounts of band-gap optimization comes from  Sigmund and  Jensen who were the first researchers to  use topology optimization approach to design a phononic crystal with maximum relative band-gap size \cite{SiJe2003}. The main idea is to find the optimal arrangement of two different materials to achieve maximum band-gap. The geometric configuration of the two materials is continually updated during designing process. The main computational challenge is the numerical solution of heterogeneous eigenvalue problems with the moving material interface.

In recent years,  there is increasing interest in investigating wave propagation in phononic materials. Numerical computation of band structures plays an essential role since wave dynamics is completely determined by the band structure of the material. Early works can be traced back to \cite{KHDD1993} where Kushwaha {\it et al.} used the plane-wave expansion to compute the band structure. The transfer matrix method was also adopted by Sigalas and  Soukoulis \cite{SiSo1995} to simulate the propagation of elastic waves through disordered solid. To date, various methods have been developed to compute the band structure of phononic crystals including the multiple scattering method\cite{KaEc1999}, the finite difference time domain method \cite{CaHL2004},   the meshless method\cite{ZZWS2016},   the (multiscale) finite element method\cite{CaRR2016, HuOs2020, LHWL2017, VeBM2013},  the homogenization method \cite{CoMa2020, ALZ2019, AFKRYZ2018},  and the singular boundary method \cite{LiCh2019}.

Among the aforementioned methods, the numerical difficulties come from two different perspectives: one is the heterogeneous nature of the phononic crystals and the other is how to efficiently impose the quasi-periodic boundary condition. However, almost all the existing methods use either indirect numerical methods such as asymptotical expansion or direct discretization using body-fitted meshes which did not work well for both numerical challenges. The mathematical analysis of finite element methods using body-fitted meshes for elliptic interface problems can be found in \cite{ChZo1998,Xu1982}.  Recently, Wang {\it et al.} \cite{WZLS2019} proposed a Petrov-Galerkin immersed finite element method to compute the band structure of the phononic crystal and imposed the quasi-periodic boundary condition directly.  However, to the best of our knowledge,  there are no numerical methods in the literature whose performance was mathematically justified.

In this paper,  we propose an unfitted Nitsche's method to compute the band structures of phononic crystal with impurities of general geometry, and prove the convergence with rigorous mathematical analysis. The heterogeneous property of the phononic crystal is modeled by the interface condition which we rewrite into a variational framework with the help of the Floquet-Bloch theory. To handle the  quasi-periodic boundary condition, the Floquet-Bloch transform is applied which reformulates the model equation with quasi-periodic boundary conditions into an equivalent model equation with periodic boundary conditions and Bloch-type interface condition.
Then, the reformulated model equations can be numerically solved by the unfitted Nitsche's type method \cite{HaHa2002, HaHa2004, HaLL2017, BCHLM2015, GuYa2018} using uniform meshes.  The proposed unfitted finite element method is motivated by our previous work of computing edge models in topological materials \cite{GuYZ2019}.
The first advantage is that it uses meshes independent of the location of the material interfaces.  It reduces the computational cost of generating body-fitted meshes, especially in designing phononic crystals.
The second advantage is that it is straightforward to impose the periodic boundary conditions since only uniform meshes are used.  Remark that imposing periodic boundary conditions on general unstructured meshes is quite technically involved, and interesting readers are referred to  \cite{VGGZ2019, AdTr2020} and the references therein about the recent development of imposing periodic boundary condition on general unstructured meshes.

As mentioned in our previous work \cite{GuYZ2019}, the discrete Nitsche's bilinear form involves the solution itself in addition to its gradient which cause the difficulties in the analysis. In this paper, we establish a solid theoretical analysis for the proposed unfitted finite element methods by conquering the above difficulties. Specifically, we show the discrete equation is well defined by using a delicate trace inequality on the cut element,  the Poincar\'e inequality between the energy norm of the original model equation and the energy norm of the modified model equation, and the explicit relation between the strain tensor and stress tensor.  By the aid of the Babu\`ska-Osborne spectral approximation theory\cite{BaOs1989, BaOs1991}, the proposed unfitted finite element method is proven to have the optimal approximation property for the eigenvalues and eigenfunctions in the high-contrast heterogeneous setting.

The paper is organized as follows. In \cref{sec:phononic}, we introduce the model of plane-wave propagation in the phononic crystals. In \cref{sec:unfitted}, we propose the unfitted numerical method to compute the band structure of phononic crystal based on the Bloch-Floquet theory and prove the proposed method admits a unique solution. In \cref{sec:error}, we carry out the optimal error analysis. In \cref{sec:example}, we present some numerical examples in a realistic setting to verify and validate our theoretical discoveries. At the end, some conclusion is draw in \cref{sec:conclusions}.

\section{Model of phononic crystal }
\label{sec:phononic}

In this section, we first present a litter digest to the two-dimensional  phononic crystal. Then we consider the model of in-plane wave propagation.

\subsection{Problem setup}
\label{ssec:setup}
Phononic crystal is designed from periodically arrangement of two different materials to achieve extraordinary properties like negative refractive index.
 The body of phononic crystal is  a kind of heterogeneous high-contrast materials.

We will mainly focus phononic crystals  with  two-dimensional Bravais lattice $\Lambda$   formed by two  primitive vectors $\bm{a}_1$ and
$\bm{a}_2$, i.e.
\begin{equation}\label{equ:lattice}
	\Lambda = \mathbb{Z}\bm{a}_1 + \mathbb{Z}\bm{a_2}
	= \left\{ m_1\bm{a}_1 + m_2\bm{a}_2:~m_1,m_2\in \mathbb{Z} \right\}.
\end{equation}
An example of square lattice  with $\bm{a}_1 = (a, 0)^T$ and $\bm{a}_2 = (0, a)^T$ is shown in \cref{fig:cirble_lattice}.
The fundamental $\Omega$ of Bravais lattice $\Lambda$ is defined as
\begin{equation}
	\Omega = \left\{  \theta_1 \bm{a}_1  +  \theta_1 \bm{a}_2: 0 \le \theta_1, \theta_2 \le 1 \right\},
\end{equation}
which is illustrated  in \cref{fig:circle_unitcell} for the square lattice.

Denote   the generating basis of  the reciprocal lattice (or dual lattice)
by $\bm{k}_i$ for $i=1,2$, which satisfy
\begin{equation}
	\bm{k}_i \cdot \bm{a}_j = 2\pi \delta_{i,j}, \quad  \forall i, j = 1, 2,
\end{equation}
where $\delta_{ij}$ is the  Kronecker delta.  Then, the reciprocal lattice
$\Lambda^{*}$ is
\begin{equation}\label{equ:duallattice}
	\Lambda^* = \mathbb{Z}\bm{k}_1 + \mathbb{Z}\bm{k_2}
	= \left\{ m_1\bm{k}_1 + m_2\bm{k}_2: ~m_1,m_2\in \mathbb{Z} \right\}.
\end{equation}
The fundament domain of the reciprocal lattice is
\begin{equation}
	\Omega^* = \left\{  \theta_1 \bm{k}_1  +  \theta_2 \mathbf{k}_2: 0 \le \theta_1, \theta_2 \le 1 \right\},
\end{equation}
which is termed as the first Brillouin zone \cite{Kittel2003}.  Again, we illustrate the the first Brillouin zone for the square lattice in \cref{fig:circle_brillouinzone}, where the triangle formed by the point $O$, $X$, and $M$ is referred as  the irreducible Brillouin zone \cite{Kittel2003}.

\begin{figure}[!h]
   \centering
   \subcaptionbox{\label{fig:cirble_lattice}}
  {\includegraphics[width=0.3106\textwidth]{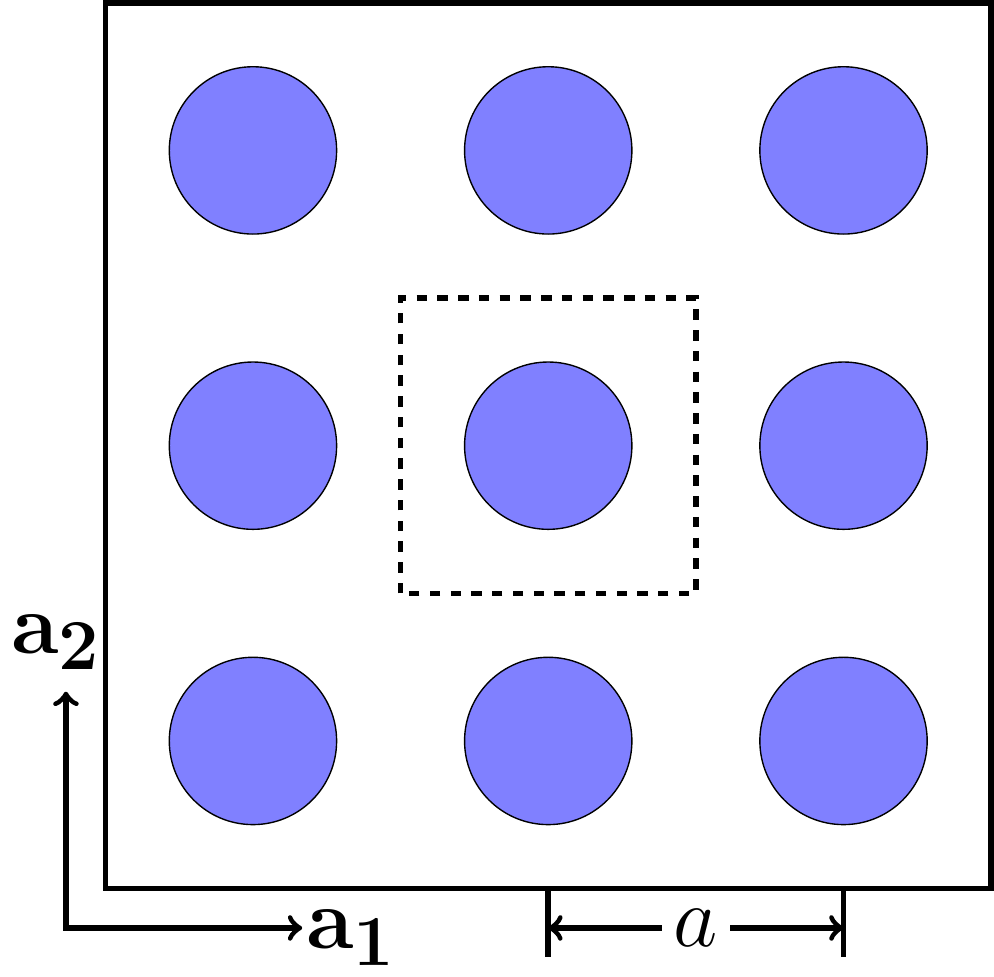}}
  \vspace{0.1in}
  \subcaptionbox{\label{fig:circle_unitcell}}
   {\includegraphics[width=0.275\textwidth]{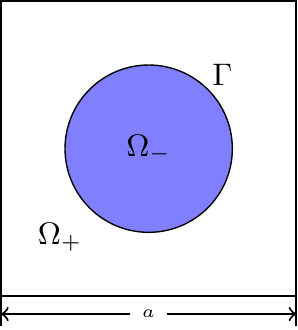}}
     \vspace{0.1in}
  \subcaptionbox{\label{fig:circle_brillouinzone}}
  {\includegraphics[width=0.306\textwidth]{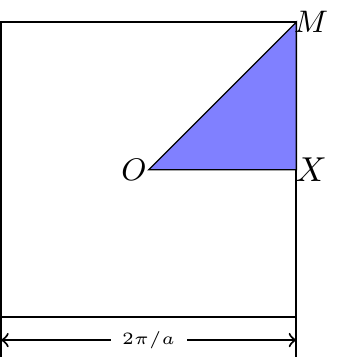}}
   \caption{ Bravais lattice $\Gamma$.  (a): 2D square lattice; (b): the unit cell;
   (c): the First Brillouin Zone}
   \label{fig:circle}
\end{figure}

The fundamental cell $\Omega$ of the phononic  crystal consists of
hard inclusion of one material $\Omega^{-}$ into a background material $\Omega^{+}$.   The background material is referred as the matrix and the inclusion is also referred as fiber.  The matrix $\Omega^+$ and the inclusion $\Omega^{-}$ are separated  by the
material interface $\Gamma$.  In \cref{fig:circle_unitcell},  we show  the fundament cell with a circular inclusion.

In this paper, we assume that both the inclusions and matrix are homogeneous isotropic elastic solids.   We use $\lambda^+$ ( or  $\lambda^-$)   denote the  first  Lam\'e parameter of matrix (or inclusion) and $\mu^+$ (or $\mu^-$) denote the first  Lam\'e parameter of matrix (or inclusion).  Similar, let  $\rho^+$ and
$\rho^-$ denote the mass density of the matrix and inclusion, respectively.  To simplify the notation, we let
\begin{equation}
\lambda =
 \begin{cases}
 	\lambda^-, &\text{in } \Omega^-,\\
 	\lambda^+, &\text{in } \Omega^+,
 \end{cases}
\quad
 \mu =
 \begin{cases}
 	\mu^-,& \text{in } \Omega^-, \\
 	\mu^+,& \text{in } \Omega^+,
 \end{cases}
 \quad
  \text{  and    }\,
   \rho =
 \begin{cases}
 	\rho^-,& \text{in } \Omega^-, \\
 	\rho^+,& \text{in } \Omega^+.
 \end{cases}
\end{equation}
For any vector-valued function  $\bm{v}$ defined on $\Omega$, let
$\llbracket \bm{v} \rrbracket $ be the jump of function $\bm{v}$ crossing the interface $\Gamma$, i.e.
\begin{equation}
  \llbracket \bm{v} \rrbracket(\bm{x})
  = \bm{v}|_{\Omega^+}(\bm{x}) - \bm{v}|_{\Omega^-}(\bm{x})
\end{equation}
for any $\bm{x}\in \Gamma$.

Throughout the paper, the standard notations for Sobolev spaces and their associated norms as in \cite{BrennerScott2008, Ciarlet2002, Evans2008}.
Given a bounded subdomain $D\subset \Omega$ and any positive integer $k$, the Sobolev space with norm
$\|\cdot\|_{k,  D} $ and seminorm $|\cdot|_{k, D}$ is denoted by $H^{k}(D)$. When $k=0$, $H^k(D)$ reduces to the standard $L^2(D)$ space.
 % Similar notations are applied to subdomains  of $\Gamma$.
 Let $(\cdot, \cdot)_D$ and $\langle \cdot, \cdot\rangle_{\Gamma}$ denote the standard $L_2$ inner products of $L_2(D)$ and
 $L_2(\Gamma)$, respectively. When $D = \Omega$, the subscript is omitted.
For a bounded domain $D= D^+\cup D^-$ with $D^+\cap D^-=\emptyset$, let  $H^{k}(D^+\cup D^-)$  be the function space consisting of piecewise Sobolev  functions $w$  such
 that $w|_{D_1}\in H^{k}(D_1)$ and $w|_{D_2}\in H^{k}(D_2)$, whose  norm  is defined as
 \begin{equation}\label{eq:pnorm}
\|w\|_{k, D^+\cup D^-} = \left( \|w\|_{k, D^+}^p + \|w\|_{k, D^-}^p\right)^{1/p},
\end{equation}
and seminorm is defined as
 \begin{equation}\label{eq:psnorm}
|w|_{k,p, D^+\cup D^-} = \left( |w|_{k, D^+}^p + |w|_{k, D^-}^p\right)^{1/p}.
\end{equation}
To avoid abusing of notation, the same notation is applied to the vector-valued function $\bm{w} = (w_1, w_2)^T$.

For any vectors $\bm{v}$ and $\bm{w}$, let $\bm{v}\otimes\bm{w}$ the tensor product of $\bm{v}$ and $\bm{w}$ and  let $\bm{v}\odot\bm{w} = \frac{1}{2}\left(\bm{v}\otimes\bm{w} + \bm{w}\otimes\bm{v}\right)$ be the symmetric tensor product.  For the quasimomentum $\bm{k}$ in the Brillouin zone, define the shift differential operator $\nabla_{\bm{k}}$ as
\begin{equation}\label{equ:shiftdiff}
	\nabla_{\bm{k}} = \nabla + \mathrm{i}\bm{k},
\end{equation}
where $\mathrm{i}$ is the imaginary unit.

In this paper, we use  the constant $C$ with or without a subscript  to denote  a generic positive constant which can be different at different occurrences. In addition, it is independent of the mesh size and the location of the interface. By $x\lesssim y$, we mean that there exists a constant C such that $x\le Cy$.

Before ending this section, we introduce some additional function spaces for Bloch-periodic (or quasi-periodic) functions
\begin{align}
	&H^k_{per}(\Omega) = \left\{ \bm{w}(\bm{x}) \in H^k(\Omega):
	\bm{w}(\bm{x}\pm\bm{a}_j) = \bm{w}(\bm{x}) \, \text{on } \partial\Omega \text{ and } j = 1, 2\right\}, \\
	&H^k_{\bm{k}}(\Omega) = \left\{ \bm{w}(\bm{x}): \exp(-\mathrm{i}\bm{k}\cdot \bm{x}) \bm{w}(\bm{x})\in H^k_{per}(\Omega)
	\right\}.
\end{align}

%\begin{align}
%	&H^k_{per}(\Lambda) = \left\{ \bm{w}(\bm{x}) \in L^2_{loc}(\mathbb{R}^2):
%	\bm{w}(\bm{x}+\bm{a}) = \bm{w}(\bm{x}), \, \forall \bm{a}\in \Lambda, \, \forall \bm{x}\in \mathbb{R}^2\right\}, \\
%	&H^k_{\bm{k}}(\Lambda) = \left\{ \bm{w}(\bm{x}): \exp(-\mathrm{i}\bm{k}\cdot \bm{x}) \bm{w}(\bm{x})\in H^k_{per}(\Lambda)
%	\right\}.
%\end{align}

\subsection{In-plane wave propagation}
The in-plane wave propagation is  modeled by the elastodynamics operator
\begin{equation}\label{equ:diffoperator}
	\mathcal{L}\bm{\phi} = - \nabla \cdot \bm{\sigma}[\bm{\phi}] =
	-\nabla \cdot \bm{\mathrm{C}}\bm{\epsilon}[\bm{\phi}].
\end{equation}
where  $\bm{\phi} = (\phi_1, \phi_2)^T$ is the displacement vector and $\bm{\mathrm{C}}$ is the  fourth-order stiffness tensor.
In \eqref{equ:diffoperator}, $\bm{\epsilon}$ is the strain tensor  which is related to the displacement via
\begin{equation}\label{equ:straintensor}
	\bm{\epsilon}[\bm{\phi}] = \nabla\odot\bm{\phi},
\end{equation}
and   $\bm{\sigma} $ is the stress tensor. For the homogeneous isotropic material,  the stress tensor and strain tensor are related by the Hook's law, i.e.
\begin{equation}\label{equ:comprelation}
	\bm{\sigma}[\bm{\phi}] = \bm{\mathrm{C}}\bm{\epsilon}[\bm{\phi}] = 2\mu\bm{\epsilon}[\bm{\phi}]
	+ \lambda \mbox{tr}(\bm{\epsilon}[\bm{u}])\mathbb{I}_2.
\end{equation}
where $\mbox{tr}(A)$ is the trace of the matrix $A$  and $\mathbb{I}_2$ is the $2\times 2$ identity matrix.

Let   $\bm{k}\in \Lambda^*$ be the quasi-momentum.  According to the Bloch theory\cite{Kittel2003},  the  in-plane wave propagation  in phononic crystal  can be reformulated to solve the following quasi-periodic eigenvalue problem   \cite{AFKRYZ2018, CoMa2020}: find $(\omega^2, \bm{\phi})\in \mathbb{R} \times H^1_{\bm{k}}(\Omega)$ such that
\begin{equation}\label{equ:blocheigen}
\begin{cases}
 \mathcal{L}\bm{\phi} =  \omega^2 \rho\bm{\phi},
&\text{in } \Omega\setminus\Gamma , \\
\llbracket \bm{\phi} \rrbracket  =
\llbracket \bm{\mathrm{C}}(\nabla\odot\bm{\phi})\bm{n}\rrbracket = 0, &\text{on }  \Gamma,
%\bm{u}(\bm{x}+\bm{a}_j) = \exp(-\mathrm{i}\bm{k}\cdot \bm{a}_j)\bm{u}(\bm{x}),
%&\text{on }  \partial\Omega\\
%		  \bm{\sigma}(\bm{x}+\bm{a}_j)\bm{n} = -\exp(-\mathrm{i}\bm{k}\cdot \bm{a}_j)\bm{\sigma}(\bm{x})\bm{n},
%		  &\text{on }  \partial\Omega,
\end{cases}
\end{equation}
where $\bm{n}$ is the unit normal vector of $\Gamma$ pointing from $\Omega^-$ to $\Omega^+$.

Due to periodicity and symmetry, we only consider the case that  the quasi-momentum  $\bm{k}$  belongs to the irreducible Brillouin zone.
For any fixed $\bm{k}\in \triangle_{OXM}$,  the eigenvalue problems \eqref{equ:blocheigen} admits a sequence of eigenvalues
$0< \omega_{\bm{k},1}^2\le\omega_{\bm{k},2}^2\le\omega_{\bm{k},3}^2\le \cdots \rightarrow \infty $ and corresponding eigenfunctions
$\bm{\phi}_{\bm{k},1}, \bm{\phi}_{\bm{k},2}, \bm{\phi}_{\bm{k},3}, \cdots$ which are orthogonal in a $\rho$-weighted $L^2_{\bm{k}}(\Omega)$ .

The existence of band gap has been mathematically justified by  Ammari et al. \cite{AmKL2009} with the asymptotic analysis in the high contrast regime. We will propose numerical method to compute the band structure in a generic setup.
\section{Unfitted Nitsche's method for computing band structure}
\label{sec:unfitted}
In this section, we are going to propose an unfitted numerical method to efficiently compute the band structure for the phononic crystal.  The numerical challenges brought by the interface eigenvalue \eqref{equ:blocheigen} is twofold: one is quasi-periodic nature of the Bloch wave and the other one is the inhomogeneity of the material. These challenges shall be discussed in the following subsections.

\subsection{Bloch-Floquet theory}
%For any fixed quasi-momentum $\bm{k}$ and displacement vector $\bm{v}$, let $\bm{k}\otimes\bm{v}$ the tensor product of $\bm{k}$ and $\bm{v}$ and  let $\bm{k}\odot\bm{v} = \frac{1}{2}\left(\bm{k}\otimes\bm{v} + \bm{v}\otimes\bm{k}\right)$ be the symmetric tensor product.
To address  the first numerical challenges, we apply the Bloch-Floquet transform $\bm{\phi}(\bm{x}) = e^{\mathrm{i}\bm{k}\cdot\bm{x}}\bm{u}(\bm{x})$.
 The quasi-periodic eigenvalue  problem can be reformulated as: find $(\omega^2, \bm{u})\in \mathbb{R} \times H^1_{per}(\Omega)$ such that
\begin{equation}\label{equ:neweigen}
	\begin{cases}
 \mathcal{L}_{\bm{k}}\bm{u} =  \omega^2 \rho\bm{u},
&\text{in } \Omega\setminus\Gamma, \\
\llbracket \bm{u} \rrbracket  =
\llbracket \bm{\mathrm{C}}(\nabla_{\bm{k}}\odot\bm{u})\bm{n}\rrbracket = 0, &\text{on }  \Gamma,
%\llbracket \bm{\mathrm{C}}(\bm{\epsilon}(\bm{u}) + \mathrm{i}\bm{k}\odot \bm{u} )\bm{n}\rrbracket = 0, &\text{on }  \Gamma,
\end{cases}
\end{equation}
%where the modified stress tensor $\bm{\hat{\sigma}}$ is defined as
%\begin{equation}\label{equ:modstree}
%	\bm{\hat{\sigma}}(\bm{u}) =
%	\bm{\mathrm{C}}(\bm{\epsilon}(\bm{u}) + \mathrm{i}\bm{k}\odot \bm{u})  = 2\mu\left(\bm{\epsilon}(\bm{u}) + \mathrm{i}\bm{k}\odot \bm{u}  \right)
%	+ \lambda \mbox{tr}(\bm{\epsilon}(\bm{u}) + \mathrm{i}\bm{k}\odot \bm{u} )I_2.
%\end{equation}
where the differential operator $\mathcal{L}_{\bm{k}}$ is defined as
\begin{equation}
	\mathcal{L}_{\bm{k}}\bm{u} = \nabla_{\bm{k}}\cdot \bm{\mathrm{C}}(\nabla_{\bm{k}} \odot\bm{u}),
\end{equation}
with $\nabla_{\bm{k}}$ being the shift differential operator defined in \eqref{equ:shiftdiff}.  We want to remark that
\begin{equation}
	\bm{\mathrm{C}}(\nabla_{\bm{k}}\odot\bm{u}) =
	2\mu \nabla_{\bm{k}}\odot\bm{u} + \lambda( \nabla_{\bm{k}}\cdot\bm{u}) \mathbb{I}_2
	\
\end{equation}
is termed as the modified stress tensor.

For a nonzero quasi-momentum $\bm{k}$, it is not difficult to see that $\mathcal{L}_{\bm{k}}$ is a self-adjoint positive define operator. The spectrum of the elastodynamics operator  $\mathcal{L}$ is the union of spectrum of $\mathcal{L}_{\bm{k}}$ for all $\bm{k}\in \Omega^*$.
Notice that the Bloch-Floquet transform $\phi(\bm{x}) = e^{\mathrm{i}\bm{k}\cdot\bm{x}}\bm{u}(\bm{x})$ is an isomorphism from  $H^1_{\bm{k}}(\Omega)$ to $H^1_{per}(\Omega)$. Then, we have the following Poincar\'e inequality:
\begin{equation}\label{equ:poincare}
	C_0[ (\bm{\mathrm{C}}(\nabla\odot\bm{u}), \nabla\odot\bm{u}) + (\bm{\mathrm{C}}(\bm{k}\odot\bm{u}), \bm{k}\odot\bm{u}) ]
	\le (\bm{\mathrm{C}}(\nabla_{\bm{k}}\odot\bm{u}), \nabla_{\bm{k}}\odot\bm{u}),
\end{equation}
where $C_0$ is a  positive constant

\subsection{Formulation of the unfitted Nitsche's method}
To find the band structure of $\mathcal{L}$, it suffices to solve a series of periodic eigenvalue problem \eqref{equ:neweigen}. The main numerical barrier is how to efficiently handle the interface condition.  We alleviate this barrier by introducing a new unfitted Nitsche's method which is seamlessly infusing with the Bloch-Floquet theory.

One merit of unfitted Nitsche's method is to use meshes independent of the location of the material interface. Due to the lattice structure of the phononic crystal, the uniform meshes is adopted. To show the main idea, we use the square lattice as the prototype model but the method works for other lattices.
We generate a  uniform mesh $\mathcal{T}_h$  on the fundamental domain $\Omega$ of the square lattice by partitioning it into $N^2$ subsquares with mesh size $h = \frac{a}{N}$ and then splitting each subsquare into isosceles right triangles, see \cref{fig:wholemesh}.

\begin{figure}[!h]
   \centering
   \subcaptionbox{\label{fig:wholemesh}}
  {\includegraphics[width=0.32\textwidth]{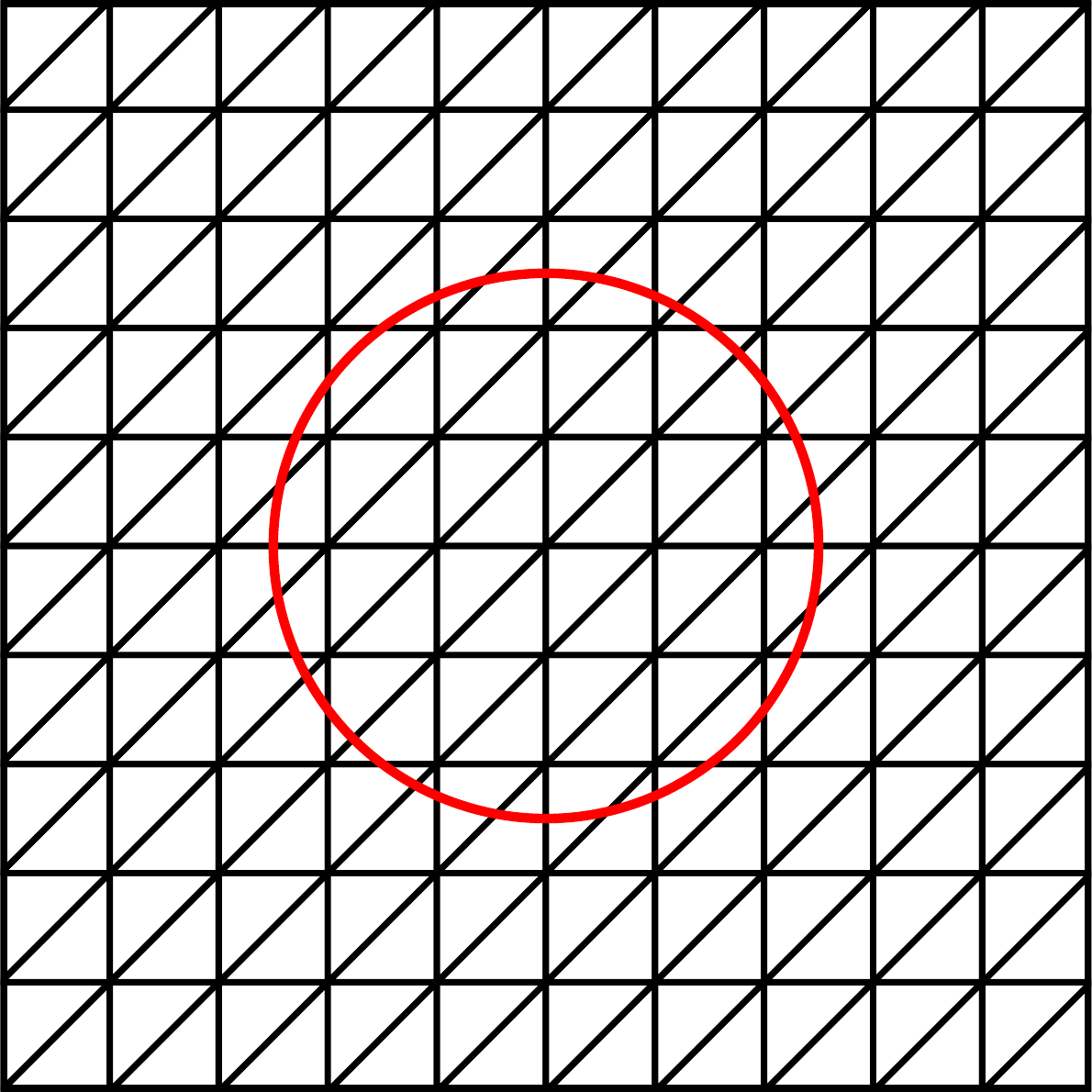}}
  \vspace{0.1in}
  \subcaptionbox{\label{fig:interiormesh}}
   {\includegraphics[width=0.32\textwidth]{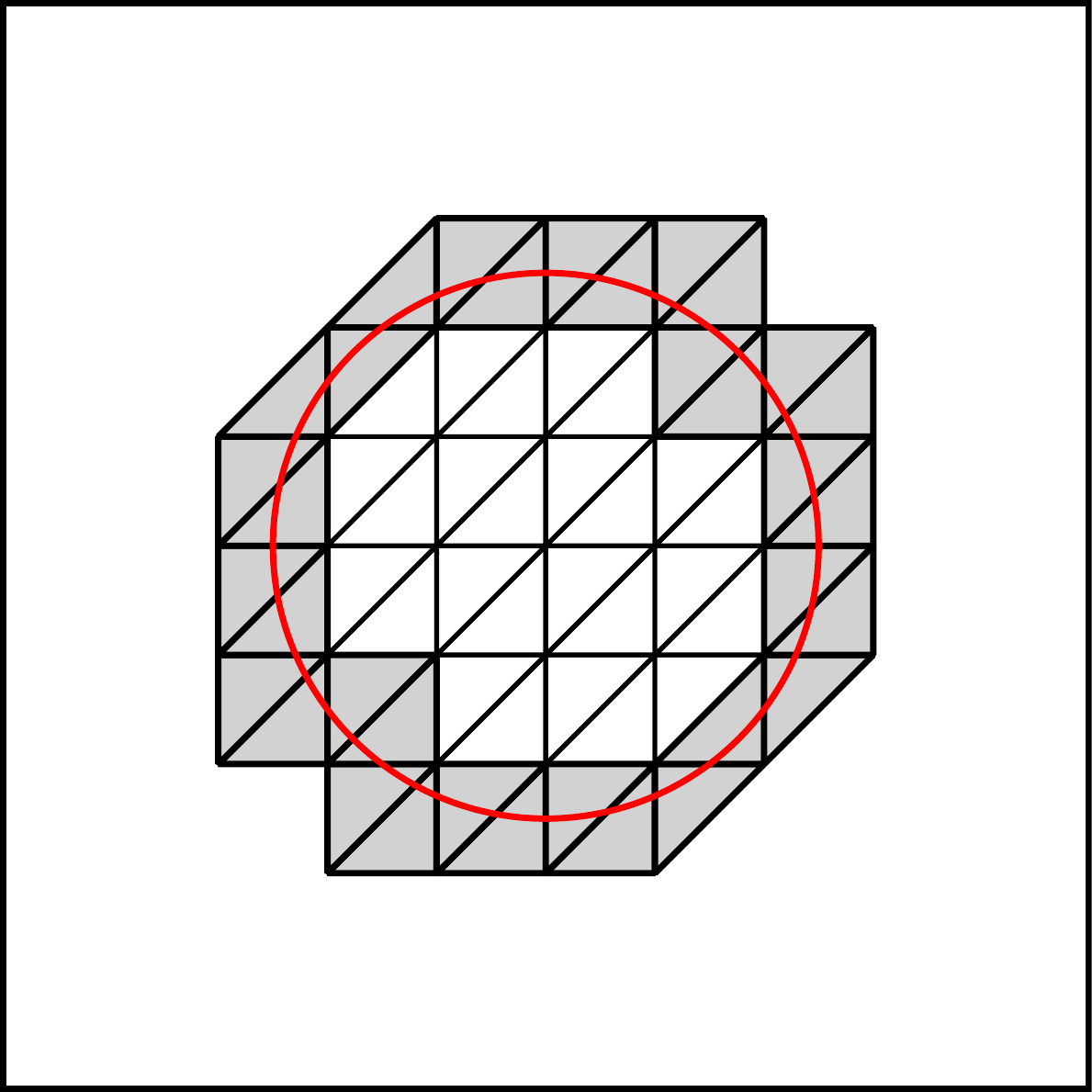}}
     \vspace{0.1in}
  \subcaptionbox{\label{fig:exteriormesh}}
  {\includegraphics[width=0.32\textwidth]{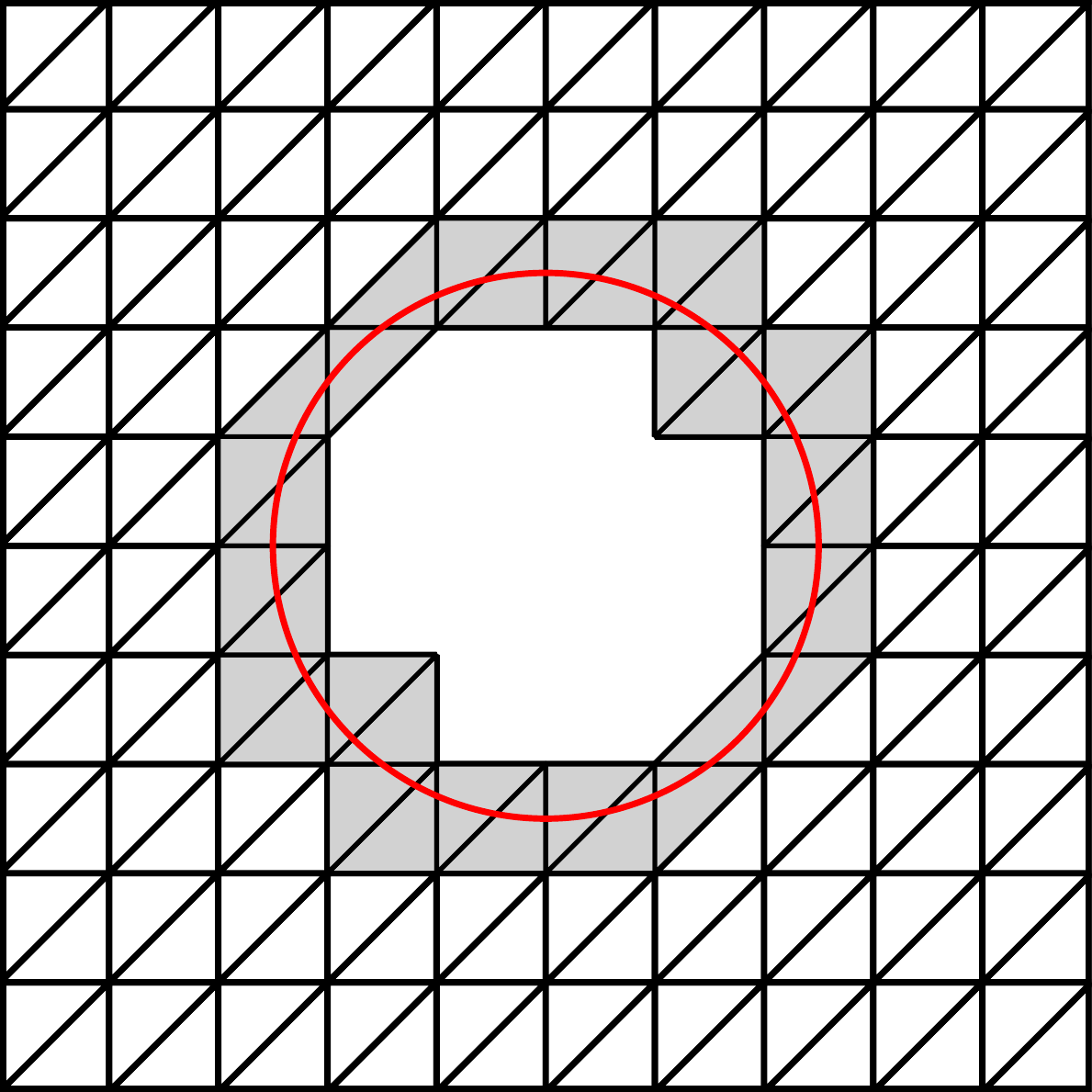}}
   \caption{Illustration of the overlapping domain decomposition of  $\Omega$  (a): Unfitted meshes on $\Omega$; (b): Subdomain $\Omega_h^-$;
   (c): Subdomain  $\Omega_h^+$.}
   \label{fig:mesh}
\end{figure}

To handle the non-smoothness of the Bloch wave across the material interface, we decompose the fundamental domain $\Omega$ into two overlapping subdomains $\Omega_{h}^+$ and $\Omega_{h}^-$
\begin{equation}
	\Omega_h^+ = \{K \in \mathcal{T}_h: K \cap \Omega^+ \neq \emptyset \}, \quad \text{   and    }\quad
	\Omega_h^- = \{K \in \mathcal{T}_h: K \cap \Omega^- \neq \emptyset \}.
\end{equation}
We illustrate the decomposition in \cref{fig:mesh}. It is undeniable that intersection of  $\Omega^+_h$ and $\Omega^-_h$  is nonempty. 	
In that sense, $\Omega_h^{\pm}$ are termed as fictitious domains.
Similarly, we can define two subtriangulations $\mathcal{T}_{h}^+$
and $\mathcal{T}_{h}^+$  as
\begin{equation}
	\mathcal{T}_h^+ = \{K \in \mathcal{T}_h: K \subset\Omega_h^+  \}, \quad \text{   and    }\quad
	\mathcal{T}_h^- = \{K \in \mathcal{T}_h: K \subset\Omega_h^- \}.\end{equation}
The common subsets of $\mathcal{T}_h^{+}$ and $\mathcal{T}_h^{-}$ is denoted by $\mathcal{T}_{\Gamma,h}$ which denotes  the set of  interface elements.

Based on the overlapping domain composition, we can definite the finite element space on each of them independently. To do this, let  $V_{h}^{s}$ ($s=\pm$)  be the standard continuous  linear finite element space on $\Omega_{h}^{s}$, i.e.
\begin{equation}
V_{h}^{s} = \left\{ \bm{v}_h \in \left[C^0(\Omega_{h}^{s})\right]^2:  \bm{v}|_{K} \in \left[\mathbb{P}_1(K)\right]^2 \text{ for any } K \in \mathcal{T}_{h}^{s}\right\},
\end{equation}
with  $\mathbb{P}_k(K)$ being  the space of polynomials of  degree   $k$ on  the element $K$.

Then the finite element space  for  the unfitted Nitsche's method is defined as  $V_h = V^+_{h} \oplus V^-_{h}$, i.e.
 \begin{equation}\label{equ:femspace}
V_h = \left\{ \bm{v}_h = (\bm{v}_{h}^+, v_{h}^-): \bm{v}_{h}^s\in V_{h}^s, \, s = \pm \right\}.
\end{equation}
To impose the periodic boundary condition,  we introduce $V_{h, per}$ as a subspace of $V_h$ which is defined as
\begin{equation}
V_{h,per} = \left\{ \bm{v}_h\in V_h: \bm{v}_h(\bm{x}\pm\bm{a}_j)=\bm{v}_h(\bm{x}) \text{ on } \partial\Omega \text{ and } j = 1, 2\right\}.
\end{equation}
Note that for the  interface element $K\in\mathcal{T}_{\Gamma,h}$,  there are two sets
of vector-valued  basis functions:  one for $V_{h}^+$ and the other for $V_{h}^-$.

For any interface element $K\in \mathcal{T}_{\Gamma,h}$,  let $K^{\pm}$   denote the part of the triangle inside $\Omega^{\pm}$  and $|K^{\pm}|$ denote the area of $K^{\pm}$.  Similarly, let $\Gamma_K=\Gamma\cap K$ be the part of $\Gamma$ in the element $K$ and let $|\Gamma_K|$ be the length of $\Gamma_K$.
Before defining the weak formulation,   it is necessary to introduce some parameters.  For $s=\pm$, let $\beta^{s} = 2\mu^{s}+\lambda^{s}$.
Define  two weights  as \cite{StLuK2020}
\begin{equation}
	\kappa^+ = \frac{\beta^-}{\beta^++\beta^-}, \quad
	\kappa^- =  \frac{\beta^+}{\beta^++\beta^-},
\end{equation}
which satisfy $\kappa^++\kappa^-=1$. Based on the two weights, we can define a weighted averaging of the displacement vector on the interface $\Gamma$  as
\begin{equation}\label{equ:wa}
	\dgal{\bm{u}} = \kappa^+ \bm{u}^+ + \kappa^-\bm{u}^-.
\end{equation}
Also, we define the stabilizing parameter for the weak formulation
\begin{equation}
	\gamma = \hat{\gamma}\frac{\beta^+\beta^-}{\beta^++\beta^-}
\end{equation}
where $\hat{\gamma}$ is a sufficiently large constant called stabilizing parameter.

Define the Nitsche's sesquilinear form $a_h(\cdot, \cdot): H^1_{per}(\Omega) \times H^1_{per}(\Omega)\rightarrow \mathbb{R}$ as
\begin{equation}\label{equ:bilinear}
\begin{split}
	a_h(\bm{u}_h, \bm{v}_h) =
	&\sum_{s = \pm}\int_{\Omega^s}
	\bm{\mathrm{C}}(\nabla_{\bm{k}}\odot \bm{u}_h): (\overline{\nabla_{\bm{k}}\odot \bm{v}_h})d\bm{x}\\
		&+\int_{\Gamma}\dgal{\bm{\mathrm{C}}\left(\nabla_{\bm{k}}\odot \bm{u}_h\right)\bm{n}}\cdot \llbracket  \overline{ \bm{v_h}} \rrbracket ds\\
	&+\int_{\Gamma} \llbracket  \bm{u_h} \rrbracket \cdot \dgal{\overline{\bm{\mathrm{C}}\left(\nabla_{\bm{k}}\odot \bm{u}_h\right)\bm{n}}}  ds \\
&+\frac{\gamma}{h} \int_{\Gamma} \llbracket  \bm{u_h}  \rrbracket \cdot
\llbracket \bm{\overline{v_h}} \rrbracket ds,
\end{split}
\end{equation}
%\begin{equation}\label{equ:bilinear}
%\begin{split}
%	a_h(\bm{u}_h, \bm{v}_h) =
%	&\sum_{s = \pm}\int_{\Omega^s}
%	\bm{\mathrm{C}}\left(\bm{\epsilon}(\bm{u}_h^s)+\mathrm{i}\bm{k}\odot \bm{u}_h^s\right): (\overline{\bm{\epsilon}(\bm{v}_h^s)-\mathrm{i}\bm{k}\odot \bm{v}_h^s})d\bm{x}\\
%		&+\int_{\Gamma}\dgal{\bm{\mathrm{C}}\left(\bm{\epsilon}(\bm{u}_h)+\mathrm{i}\bm{k}\odot \bm{u}_h\right)\bm{n}}\cdot \llbracket  \overline{ \bm{v_h}} \rrbracket ds\\
%		&+\int_{\Gamma} \llbracket  \bm{u_h} \rrbracket \cdot \dgal{\overline{\bm{\mathrm{C}}\left(\bm{\epsilon}(\bm{v}_h)+\mathrm{i}\bm{k}\odot \bm{u}_h\right)\bm{n}}}  ds \\
%		&+\frac{\gamma}{h} \int_{\Gamma} \llbracket  \bm{u_h}  \rrbracket \cdot
%		\llbracket \bm{\overline{v_h}} \rrbracket ds,
%\end{split}
%\end{equation}
where $h$ is the mesh size,  $\bm{\mathrm{C}}$ is the fourth-order stiffness tensor, and  $A:B$ is the  Frobenius inner product  of two matrices $A$ and $B$.

Given a quasi-momentum $\bm{k}$ in the  Brillouin zone,  the unfitted Nitsche's method for the eigenvalue problem \eqref{equ:neweigen} is to find  the eigenpair $(\omega_h^2, \bm{u}_h) \in \mathbb{R}\times V_{h, per}$ such that
\begin{equation} \label{equ:nit}
	a_h(\bm{u}_h, \bm{v}_h) =
	\omega_h^2 b(\bm{u}_h, \bm{v}_h), \quad\forall
	 \,\bm{v}_h \in  V_{h, per},
\end{equation}
where
\begin{equation}
	b(\bm{u}_h, \bm{v}_h) =
	\int_{\Omega}
	\rho\bm{u}_h\overline{\bm{v}_h}d\bm{x}.
\end{equation}

\begin{remark}\label{rmk:equiv}
	Using the definition of the fourth-order stiffness tensor $\bm{\mathrm{C}}$,  we can write the Nitsche's sesquilinear form $a_h(\cdot, \cdot)$ into the following equivalent form
	\begin{equation}\label{equ:altbilinear}
\begin{split}
	&a_h(\bm{u}_h, \bm{v}_h)\\
	 = &2\mu\sum_{s = \pm}\int_{\Omega^s}
	(\nabla_{\bm{k}}\odot \bm{u}_h): (\overline{\nabla_{\bm{k}}\odot \bm{v}_h})d\bm{x}
	+2\lambda\sum_{s = \pm}\int_{\Omega^s}
	(\nabla_{\bm{k}}\cdot \bm{u}_h)(\overline{\nabla_{\bm{k}}\cdot \bm{v}_h})d\bm{x}\\
		&+\int_{\Gamma}\dgal{2\mu(\nabla_{\bm{k}}\odot
		\bm{u}_h)\bm{n} +
		\lambda(\nabla_{\bm{k}}\cdot \bm{u}_h)\bm{n}}\cdot \llbracket  \overline{ \bm{v_h}} \rrbracket ds\\
		&+\int_{\Gamma} \llbracket  \bm{u_h} \rrbracket \cdot \dgal{\overline{2\mu(\nabla_{\bm{k}}\odot
		\bm{v}_h)\bm{n} +
		\lambda(\nabla_{\bm{k}}\cdot \bm{v}_h)\bm{n}}} ds \\
		&+\frac{\gamma}{h} \int_{\Gamma} \llbracket  \bm{u}_h  \rrbracket \cdot
		\llbracket \bm{\overline{\bm{v}_h}} \rrbracket ds.
\end{split}
\end{equation}
From this equivalent expression,  it is not difficult to see  that $a_h(\cdot, \cdot)$  is a  symmetric sesquilinear form and  the eigenvalues $\omega_h^2$ are real.
\end{remark}

\begin{remark}
	To make the method be more robust with respect to small element cut, we can  adopt the ghost penalty technique \cite{HaLL2017, Burm2010} to add more stabilizing
	terms in the vicinity of cut element.
\end{remark}

\subsection{Well-posedness of unfitted Nitsche's method} This subsection is devoted to  establishing  the well-posedness of the proposed unfitted Nitsche's method \cref{equ:nit}.  We start it with by showing the following consistency results:

\begin{theorem}
	Let $(\omega^2, \bm{u})$ be the eigenpair of the eigenvalue problem \cref{equ:neweigen}. Then,  we have
	\begin{equation} \label{equ:consist}
		a_h( \bm{u},  \bm{v}) = \omega^2 b_h( \bm{u},  \bm{v}), \quad \forall\,  \bm{v} \in H^1_{per}(\Omega).
	\end{equation}
\end{theorem}
\begin{proof}
 The solution $\bm{u}$ satisfies $\llbracket \bm{u} \rrbracket  =
\llbracket \bm{\mathrm{C}}\left(\nabla_{\bm{k}}\odot\bm{u} \right)\bm{n}\rrbracket = 0, \text{on }  \Gamma$. Using this fact and the Green's formulation,  We can derive the Nitsche's weak formulation via the same technique as Nitsche's  method for general boundary conditions as in \cite{JuSt2009}.
\end{proof}

Taking $\bm{v}_h$ as a function in the unfitted Nitsche's finite element space $V_{h, per}$ in \cref{equ:consist}, it is straightforward to verify  that
\begin{equation}\label{equ:gal}
	a_h(\bm{u}-\bm{u}_h, \bm{v}_h) = 0, \quad \forall
	\bm{v}_h \in V_{h, per},
\end{equation}
which is termed as the Galerkin orthogonality.

We are now in a position to show the stability of the unfitted Nitsche's method. Before that, we need to introduce some norms.  For any quasi-momentum in the  Brillouin zone,  we introduce the following  norm
\begin{equation}\label{equ:norm}
\begin{split}
	\VERT\bm{v}_h\VERT^2 = &\sum_{s=\pm}\left(\bm{\mathrm{C}}
	\left(\nabla_{\bm{k}}\odot
	\bm{v}_h \right) ,  \nabla_{\bm{k}}\odot
	\bm{v}_h
	\right)_{\Omega^s} +  \\
		&\sum_{K\in\mathcal{T}_{\Gamma,h}}h
	\|
	\dgal{\bm{\mathrm{C}}(\nabla_{\bm{k}}\odot
	\bm{v}_h)\bm{n} }\|_{0, \Gamma_K}^2+\\
	&
	\sum_{K\in\mathcal{T}_{\Gamma,h}}\frac{\gamma}{h}\|
	\llbracket \bm{v}_h \rrbracket\|_{0, \Gamma_K}^2.
	\end{split}
\end{equation}

To show the well-posedness of the unfitted Nitsche's method, we need several technical lemmas.   We begin with the trace inequality.
\begin{lemma}\label{lem:traceinequal}
Let $\bm{v}_h$ be a finite element function in $V_{per,h}$ and $\bm{k}$ be a quasi-momentum in the  Brillouin zone.  For $s = \pm$, the following inequalities hold:
 \begin{align}
&\| \dgal{\bm{\mathrm{C}}(\bm{k}\odot \bm{v}_h)\bm{n}}\|_{0, \Gamma_K}^2  \le C_1h^{-1} \| \bm{\mathrm{C}}(\bm{k}\odot \bm{v}_h)\|_{0, K^+\cup K^-}^2, \label{equ:l2tr}\\
& \|\dgal{\bm{\mathrm{C}}(\nabla\odot \bm{v}_h)\bm{n}}\|_{0, \Gamma_K}^2  \le C_2h^{-1} \| \bm{\mathrm{C}}(\nabla\odot \bm{v}_h)\|_{0, K^+\cup K^-}^2. \label{equ:h1tr}
\end{align}
\vspace{-0.2in}
\end{lemma}
\begin{proof}
	The proof of \eqref{equ:h1tr} is based on the fact that $\nabla\odot\bm{u}^s_h$ is constant which can be found in \cite{HaHa2004}.  To show \eqref{equ:l2tr}, we use the following inequality for each component of the vector-valued function $\bm{v}_h^s
=(v_{h,1}^s, v_{h,2}^s)^T$
\begin{equation}\label{equ:temptrace}
	\|v_{h,i}^s\|_{0, \Gamma_K}^2\le C_3 h \|v_{h,i}^s\|_{0, K^s}^2
\end{equation}
for $i=1,2$ and $s=\pm$. The inequality \eqref{equ:temptrace} is  proved in
\cite[Lemma 3.1]{WuXi2019}.  Let
\begin{equation*}
	\bm{\mathrm{C}}(\bm{k}\odot \bm{v}_h^s)
	=\begin{pmatrix}
		w_{11}^s&w_{12}^s\\
		w_{11}^s&w_{11}^s
	\end{pmatrix}.
\end{equation*}
Using \eqref{equ:temptrace}, we deduce that
\begin{equation}
\begin{split}
		 \| \bm{\mathrm{C}}(\bm{k}\odot \bm{v}_h^s)\|_{0, \Gamma_K}^2
	 = &\sum_{i,j=1}^2 \| w_{ij}^s\|_{0, \Gamma_K}^2
	 \le C_3 h^{-1} \|w_{ij}^s\|_{0, K^s}^2 \\
	 =&C_3h^{-1}
	 \| \bm{\mathrm{C}}(\bm{k}\odot \bm{v}_h)\|_{0, K^s}^2.
\end{split}
\end{equation}
By the definition of the weighted averaging \eqref{equ:wa} and the fact $\kappa^{\pm}\le 1$, we obtain
that
\begin{equation*}
\begin{split}
	\|\dgal{\bm{\mathrm{C}}(\bm{k}\odot \bm{v}_h)\bm{n}}\|_{0, \Gamma_K}^2  & \le 2\| \bm{\mathrm{C}}(\bm{k}\odot \bm{v}_h^+)\|_{0, \Gamma_K}^2 + 2\| \bm{\mathrm{C}}
	(\bm{k}\odot \bm{v}_h^-)\|_{0, \Gamma_K}^2\\
	 &\le 2C_3h^{-1}
	 \left(\| \bm{\mathrm{C}}(\bm{k}\odot \bm{v}_h^+)\|_{0, K^+}	^2
	 +\|\bm{\mathrm{C}}(\bm{k}\odot \bm{v}_h^-)\|_{0, K^-}	^2\right)\\
	 &= 2C_3h^{-1}\| \bm{\mathrm{C}}(\bm{k}\odot \bm{v}_h)\|_{0, K^+\cup K^-}^2,
	\end{split}
\end{equation*}
which completes the proof of \eqref{equ:l2tr} with $C_1=2C_3$.
\end{proof}

Next, we establish the following relationship between the strain and stress tensor
\begin{lemma}\label{lem:invrel}
	Let $\bm{\mathrm{C}}$ be the fourth-order stiffness tensor \eqref{equ:comprelation} and $A$ be any symmetric second-order tensor.  Then, the following inequality holds
\begin{equation}\label{equ:trone}
		\bm{\mathrm{C}}A:\bm{\mathrm{C}}A \le (4\mu+2\lambda) \bm{\mathrm{C}}A:A.
	\end{equation}
\end{lemma}
\begin{proof}
	By the definition of fourth-order stiffness tensor \eqref{equ:comprelation}, it follows that
	\begin{equation}\label{equ:tracerel}
		\mbox{tr}(\bm{\mathrm{C}}A) = \mbox{tr}(2\mu A + \lambda \mbox{tr}(A)\mathbb{I}_2)
		 = (2\mu+2\lambda)\mbox{tr}(A).
	\end{equation}
Notice that
\begin{equation}\label{equ:trtwo}
	\bm{\mathrm{C}}A:A =(2\mu A + \lambda \mbox{tr}(A)\mathbb{I}_2):A
	= 2\mu A:A + \lambda \mbox{tr}(A)^2 \ge \lambda \mbox{tr}(A)^2,
\end{equation}
where we have used the fact $A:A\ge 0$.  Using \eqref{equ:trone} and \eqref{equ:trtwo}, we can deduce that
\begin{equation}
	\bm{\mathrm{C}}A:\bm{\mathrm{C}}A = 2\mu(\bm{\mathrm{C}}A:A)
	+ \lambda \mbox{tr}(\bm{\mathrm{C}}A)\mbox{tr}(A) \le (4\mu+2\lambda)\bm{\mathrm{C}}A:A.
\end{equation}
\end{proof}

With the preparations, we are ready to show our main result

\begin{theorem}\label{thm:wellposed}
	Let $\bm{k}$ be a nonzero quasi-momentum in the  Brillouin zone.  Suppose the stabilizing parameter $\hat{\gamma}$ is large enough. Then, there exist $C_4, C_5>0$ such that the following continuity and coercivity results hold
	\begin{align}
		& a_h(\bm{u}_h, \bm{v}_h) \le C_4 \VERT \bm{v}_h \VERT \VERT \bm{u}_h \VERT,\label{equ:cont}\\
		& a_h(\bm{u}_h, \bm{u}_h) \ge C_5\VERT \bm{u}_h \VERT^2. \label{equ:coer}
	\end{align}
\end{theorem}
\begin{proof}
	The continuity \eqref{equ:cont} is a direct implication the definition of  the  sesquilinear form \eqref{equ:bilinear} and the Cauchy-Schwartz inequality. It suffices to show the coercivity \eqref{equ:coer}.
	Letting  $\bm{v}_h=\bm{u}_h$ in \eqref{equ:bilinear} and applying the Young's inquality with $\epsilon$ imply
	\begin{equation*}
		\begin{split}
			a_h(\bm{u}_h, \bm{u}_h) =
	&\sum_{s = \pm}\int_{\Omega^s}
	\bm{\mathrm{C}}(\nabla_{\bm{k}}\odot \bm{u}_h): (\overline{\nabla_{\bm{k}}\odot \bm{u}_h})d\bm{x} + \frac{\gamma}{h} \int_{\Gamma} \llbracket  \bm{u_h}  \rrbracket \cdot
\llbracket \bm{\overline{u_h}} \rrbracket ds \\
		&+2\mbox{Re}\int_{\Gamma}\dgal{\bm{\mathrm{C}}\left(\nabla_{\bm{k}}\odot \bm{u}_h\right)\bm{n}}\cdot \llbracket  \overline{ \bm{u_h}} \rrbracket ds\\
 \ge & \sum_{s = \pm}\left(
	\bm{\mathrm{C}}(\nabla_{\bm{k}}\odot \bm{u}_h), \nabla_{\bm{k}}\odot \bm{u}_h \right)_{\Omega^s}
	  + \frac{\gamma-2\epsilon}{h} \|\llbracket  \bm{u_h}  \rrbracket\|_{0, \Gamma}^2  \\
	  & - \frac{h}{\epsilon}
	 \| \dgal{\bm{\mathrm{C}}\left(\nabla_{\bm{k}}\odot \bm{u}_h\right)\bm{n}}\|_{0, \Gamma}^2 \\
	 =: &  I_1 + I_2 - I_3.
		\end{split}
	\end{equation*}

Notice $I_1$ and $I_2$ are included in the mesh-dependent norm $ \VERT \cdot \VERT$. We only need to estimate $I_3$. Let $\mu_m = \max(\mu^+, \mu^-)$ and $\lambda_m = \max(\lambda^+, \lambda^-)$.   Using Lemma  \ref{lem:traceinequal} and Lemma \ref{lem:invrel}, we deduce that
\begin{equation*}
		\begin{split}
		I_3 =& \frac{h}{\epsilon}
	 \| \dgal{\bm{\mathrm{C}}\left(\nabla_{\bm{k}}\odot \bm{u}_h\right)\bm{n}}\|_{0, \Gamma}^2 \\
	 \le & \frac{2h}{\epsilon}
	 \| \dgal{\bm{\mathrm{C}}\left(\nabla\odot \bm{u}_h\right)\bm{n}}\|_{0, \Gamma}^2  +  \frac{2h}{\epsilon}
	 \| \dgal{\bm{\mathrm{C}}\left(\bm{k}\odot \bm{u}_h\right)\bm{n}}\|_{0, \Gamma}^2 \\
	 \le & \frac{2C_1}{\epsilon}
	 \| \bm{\mathrm{C}}(\nabla \odot \bm{u}_h)\|_{0,
	  \Omega^+\cup \Omega^-}^2
	 + \frac{2C_2}{\epsilon}
	 \| \bm{\mathrm{C}}(\bm{k}\odot \bm{u}_h)\|_{0,
	  \Omega^+\cup \Omega^-}^2\\
	  \le & \frac{2C_1(4\mu_m+ 2\lambda_m)}{\epsilon}
	  \sum_{s = \pm}\left(
	\bm{\mathrm{C}}(\nabla\odot \bm{u}_h), \nabla \odot \bm{u}_h \right)_{\Omega^s} + \\
	&\frac{2C_2(4\mu_m+ 2\lambda_m)}{\epsilon}
	  \sum_{s = \pm}\left(
	\bm{\mathrm{C}}(\bm{k}\odot \bm{u}_h), \bm{k} \odot \bm{u}_h \right)_{\Omega^s} \\
	\le & \frac{2C_4(4\mu_m+ 2\lambda_m)}{\epsilon}I_1,
	%\sum_{s = \pm}\left(
	%\bm{\mathrm{C}}(\nabla_{\bm{k}}\odot \bm{u}_h),
	%\nabla_{\bm{k}}\odot \bm{u}_h \right)_{\Omega^s},
		\end{split}
	\end{equation*}	
	where we have used the Poincar\'e inequality \eqref{equ:poincare}  and $C_4 = \max(C_0C_1, C_0C_2)$.

	Combining the above two estimates, we have
	\begin{equation*}
		\begin{split}
			&a_h(\bm{u}_h, \bm{u}_h)\\
			\ge
  & (1-\frac{2C_4(4\mu_m+ 2\lambda_m)}{\epsilon})
 \sum_{s = \pm}\left(
	\bm{\mathrm{C}}(\nabla_{\bm{k}}\odot \bm{u}_h), \nabla_{\bm{k}}\odot \bm{u}_h \right)_{\Omega^s}
	+
	   \frac{\gamma-2\epsilon}{h} \|\llbracket  \bm{u_h}  \rrbracket\|_{0, \Gamma}^2  \\
	   =&   (1-\frac{3C_4(4\mu_m+ 2\lambda_m)}{\epsilon})
 \sum_{s = \pm}\left(
	\bm{\mathrm{C}}(\nabla_{\bm{k}}\odot \bm{u}_h), \nabla_{\bm{k}}\odot \bm{u}_h \right)_{\Omega^s}
	  +\\
	 & \frac{h}{\epsilon}
	 \| \dgal{\bm{\mathrm{C}}\left(\nabla_{\bm{k}}\odot \bm{u}_h\right)\bm{n}}\|_{0, \Gamma}^2 +
	   \frac{\gamma-2\epsilon}{h} \|\llbracket  \bm{u_h}  \rrbracket\|_{0, \Gamma}^2 .
		\end{split}
	\end{equation*}
Taking $\epsilon = 6C_4(4\mu_m+2\lambda_m)$ and $\hat{\gamma} \ge 4\epsilon\frac{\beta^+\beta^-}{\beta^+\beta^-}$  concludes the proof of \eqref{equ:coer}.
\end{proof}

From \cref{thm:wellposed},  we can see the discrete sesquilinear form \eqref{equ:bilinear} is continuous and coercive  with respect to the mesh-dependent norm \eqref{equ:norm}.  The Lax-Milgram theorem implies the unfitted Nitsche's method \eqref{equ:nit} is well-posed. The spectral theory says the discrete eigenvalue of \eqref{equ:nit} can be listed as
\begin{equation}\label{equ:eigenvalue}
	0 < \omega^2_{h,1} \le \omega^2_{h,2} \le \cdots \le \omega^2_{h,n_h},
\end{equation}
and the corresponding $L^2$ eigenfunctions are $\bm{u}_{h,1}, \bm{u}_{h,2}, \ldots, \bm{u}_{h,n_h} $ where $n_h$ is the dimension of the Nitsche's finite element space $V_{h, per}$.

\section{Error estimates}
\label{sec:error}
In this section, we shall conduct the  error analysis for the proposed unfitted  Nitsche's  method \eqref{equ:nit} using the Babu\v{s}ka-Osborn theory. To prepare the error analysis, we  introduce an  extension operator $X^{s}$ ($s= \pm$)  to extend an $H^2$ function defined on a subdomain $\Omega^{s}$ to  the fundamental cell $\Omega$.  For a function $\bm{v}\in H^2(\Omega^s)$, the extended function $X^{s}\bm{v} \in H^2(\Omega)$ is defined to satisfy
\begin{equation}
	(X^s\bm{v})|_{\Omega^s} = \bm{v},
\end{equation}
and
\begin{equation}
	\|X^s\bm{v}\|_{r, \Omega} \le C \|\bm{v}\|_{r, \Omega^s}, \quad r = 0, 1, 2,
\end{equation}
for $s = \pm$.

For $s=\pm$, let $\pi_{h}^s$ be the Scott-Zhang interpolation operator \cite{ScZh1990} on  $H^1(\Omega^s_h)$. The interpolation operator on  the finite element space $V_{h, per}$  is  defined as
\begin{equation}\label{equ:interp}
	I_h\bm{v} = (\pi_{h}^+X^{+}\bm{v}, \pi_{h}^-X^{-}\bm{v})\in V_h.
\end{equation}
Using the same argument as in  \cite{HaLL2017}, we can establish the following approximation property in the mesh-depending norm \eqref{equ:norm}
\begin{equation}\label{equ:interperror}
	\VERT \bm{v} - I_h\bm{v}\VERT \lesssim h\|\bm{v}\|_{2, \Omega^+\cup \Omega^-}.
\end{equation}

To adopt the Babu\v{s}ka-Osborn theory\cite{BaOs1989,BaOs1991} , we define  solution operator $T: L^2(\Omega)\rightarrow H^1(\Omega^+\cup\Omega^-)$  as
\begin{equation}\label{equ:soloperator}
	a_h(T\bm{f},  \bm{v}) = b(\bm{f}, \bm{v}), \quad \forall \bm{v}
	\in H^1(\Omega^+\cup\Omega^-)
\end{equation}
for any $\bm{f}\in L^2(\Omega)$.  The interface eigenvalue problem  \eqref{equ:neweigen} can be reinterpreted as
\begin{equation}
	T\bm{u} = \zeta\bm{u},
\end{equation}
where $\zeta^{-1} = \omega^2$.

In a similar way, we can define the discrete solution operator $T_h: L^2(\Omega) \rightarrow V_{h, per}$ as
\begin{equation}\label{equ:dissoloperator}
	a_h(T_h\bm{f},  \bm{v}_h) = b(\bm{f}, \bm{v}_h),
	 \quad \forall \bm{v} \in V_{h, per},
\end{equation}
for any $\bm{f}\in L^2(\Omega)$. Using the discrete solution operator,  the discrete eigenvalue problem\eqref{equ:nit} is equivalent to
\begin{equation}
	T_h\bm{u}_h = \zeta_h\bm{u}_h,
\end{equation}
where $\zeta^{-1}_h = \omega^2_h$.

From the definition \eqref{equ:bilinear}, the sesquilinear form $a_h(\cdot, \cdot)$ is Hermitian which implies both $T$ and $T_h$ are self-adjoint.  It is also note that  $T$ and $T_h$ are compact operators from $L^2(\Omega)$ to $L^2(\Omega)$.
For the solution operators $T$ and $T_h$, we have the following approximation operator:
\begin{theorem}\label{thm:solopapp}
	Let $T$  be the solution operator defined in \eqref{equ:soloperator} and $T_h$ be the discrete solution operator define in \eqref{equ:dissoloperator}.  Then, we have
  \begin{align}
  	&\VERT T\bm{f} - T_h\bm{f} \VERT
  	\le C h\|\bm{f}\|_{0, \Omega}, \label{equ:energyerror}\\
   &\|T\bm{f}- T_h\bm{f}\|_{0, \Omega}
    \le C h^2\|\bm{f}\|_{0, \Omega}\label{equ:l2error}.
  \end{align}
\end{theorem}
\begin{proof}
The error estimate \eqref{equ:energyerror} can be proved by combining orthogonality, continuity, and the coercivity  of the sesquilinear form $a_h(\cdot, \cdot)$  and the approximation property of the interpolation operator \eqref{equ:interp}.   The error estimate \eqref{equ:l2error} can be established using the Aubit-Nitsche' argument \cite{BrennerScott2008, Ciarlet2002}.
\end{proof}

As a direct consequence of  Theorem \cref{thm:solopapp},  it is straightforward to show that
\begin{equation} \label{equ:operatorapproximation}
	\|T-T_h\|_{\mathcal{L}(L^2(\Omega))}  \le C h^2.
\end{equation}

   Denote the resolvent set of the operator $T$ (or $T_h$) by $\sigma(T)$ (or $\sigma(T_h)$ and the spectrum set of the operator $T$ (or $T_h$ ) by $\rho(T)$ (or $\rho(T_h)$).
   Suppose $\mu$ is an eigenvalue of the compact operator $T$  with algebraic multiplicities $m$. Let $\mathcal{C}$ be a circle in the complex plane centered at $\zeta$  which is contained  in  resolvent set of $T$ and encloses  no other spectrum  points of $T$.
  When $h$ is sufficiently small,  $\mathcal{C}$  is also contained in the resolvent set of $T_h$.  We define the Reisz spectral projection associated with $T$ and $\mathcal{C}$   as   \begin{equation}
E = \frac{1}{2\pi\mathrm{i}} \int_{\mathcal{C} } (z-T)^{-1}dz.
\end{equation}
and the discrete analogue as
 \begin{equation}
E_h = \frac{1}{2\pi\mathrm{i}} \int_{\mathcal{C} } (z-T_h)^{-1}dz.
\end{equation}
 According to  \cite{BaOs1991},  $E$ is a projection onto the space of generalized eigenvectors associated with $\zeta$ and $T$.

   Then, the Babu\v{s}ka-Osborn theory\cite{BaOs1989,BaOs1991} tells us that
\begin{theorem}\label{thm:nopol}
Let $\zeta$ be an eigenvalue of $T$ with algebraically multiplicity $m$ and  $\mathcal{C}$ be the circle defined above. Then, for sufficiently
small $h$, there following statements hold.
\begin{enumerate}%[label=(\roman*).]
	\item There are exactly $m$ eigenvalues
	 $\zeta_{h,1}, \cdots \zeta_{h,m}$ of $\mathcal{T}$ enclosed  in $\Gamma$. Furthermore, $\lim\limits_{h\rightarrow 0}\zeta_{h,j} = \zeta$.
	 \item $E_h$ is a onto projection to the direct sum of the spaces of eigenvectors corresponding to these eigenvalues  $\zeta_{h,1}, \cdots \zeta_{h,m}$ of $\mathcal{T}$.
	 \item There is a constant C independent of h such that
 \begin{equation}
 	\hat{\delta}(R(E), R(E_h)) \le C \|(T-T_h)|_{R(E_h)}\|,
 \end{equation}
 where $\hat{\delta}(R(E), R(E_h))$ is the gap between the range of $E$ and the range of $E_h$ and    $(T-T_h)|_{R(E_h)}$ is the restriction of
 $T-T_h$ to $R(E)$.
\end{enumerate}
\end{theorem}

Now, we are in the position to present our main results on the numerical approximation of the eigenvalues and eigenfunctions:
\begin{theorem} \label{thm:eigapp}
Let $\mu_h$ be an eigenvalue of $T_h$ satisfying  $\lim\limits_{h\rightarrow 0} \zeta_h = \zeta$.  Let $\bm{u}_h$ be a unit eigenvector of $T_h$
 corresponding to the eigenvalue $\zeta_h$.  Then there exists a unit eigenvector $\bm{u}\in R(E)$ such that the following estimates hold
\begin{align}
 &\|\bm{u}-\bm{u}_h\|_{0, \Omega} \le  Ch^2 \|\bm{u}\|_{2, \Omega^+\cup\Omega^-},  \label{equ:efunerr}\\
 &  |\zeta - \zeta_h| \le Ch^2 , \label{equ:evalerr}\\
 & |\omega^2 - \omega^2_h| \le Ch^2 . \label{equ:evalueerr}
\end{align}

\end{theorem}

\begin{proof}
First,  we consider the approximation capability in the eigenfunction. To do this, we approximation theory of abstract compact operator in \cite{BaOs1991}.  Theorem ~7.4 in \cite{BaOs1989} implies that
\begin{align*}
\|\bm{u}-\bm{u}_h\|_{0, \Omega} & \le \|(T-T_h)|_{R(E)}\|_{0, \Omega} = \sup_{\substack{ \bm{v}\in R(E)\\
 \VERT \bm{v}\VERT_h = 1}} \|T \bm{v} - T_h \bm{v}\|_{0, \Omega}
 \le Ch^2 \|\bm{u}\|_{2, \Omega^+\cup\Omega^-},
\end{align*}
where we have used the approximation property \eqref{equ:operatorapproximation}.  This completes the proof  \eqref{equ:efunerr}.

We then turn to the estimate  \eqref{equ:evalerr}. Without loss of generality, suppose  $\bm{u}_1$, \ldots, $\bm{u}_m$ form the unit  basis for $R(E)$. Again, the Theorem ~7.3 in \cite{BaOs1991}  implies that there exists a constant $C$ such that
 \begin{equation}\label{equ:eigapp}
|\zeta - \zeta_h| \le C \sum_{j,k = 1}^m |((T-T_h)\bm{u}_j, \bm{u}_k)| + C\|(T-T_h)|_{R(E)}\|_{0, \Omega}^2.
\end{equation}
From \eqref{equ:operatorapproximation},  the second term in \eqref{equ:eigapp} is bounded  above by $\mathcal{O}(h^2)$.
 We only need to estimate  the first term in \eqref{equ:eigapp}.
 By   \eqref{equ:soloperator}, \eqref{equ:dissoloperator} and
 the Galerkin orthogonality \eqref{equ:gal}, we obtain that
 \begin{equation}
\begin{split}
  ((T-T_h)\bm{u}_j, \bm{u}_k) =& (\bm{u}_j, (T-T_h)\bm{u}_k)\\
  = &a_h(T\bm{u}_j, T\bm{u}_k - T_h\bm{u}_k)\\
  = &a_h(T\bm{u}_j-T_h\bm{u}_j, T\bm{u}_k - T_h\bm{u}_k) + a_h(T_h\bm{u}_j, T\bm{u}_k - T_h\bm{u}_k)\\
    = &a_h(T\bm{u}_j-T_h\bm{u}_j, T\bm{u}_k - T_h\bm{u}_k) + \overline{a_h(T\bm{u}_k - T_h\bm{u}_k, T_h\bm{u}_j)}\\
    = &a_h(T\bm{u}_j-T_h\bm{u}_j, T\bm{u}_k - T_h\bm{u}_k)\\
    \le &C\VERT T\bm{u}_j-T_h\bm{u}_j\VERT_h
    \VERT T\bm{u}_k-T_h\bm{u}_k\VERT_h\\
    \le& Ch^2\|\bm{u}_j\|_{2, \Omega^+\cup\Omega^-}\|\bm{u}_k\|_{2, \Omega^+\cup\Omega^-}\\
    \le& Ch^2.
\end{split}
\end{equation}
Combining the above two estimates gives the optimal approximation property of the eigenvalue \eqref{equ:evalerr}.

To show the last estimate \eqref{equ:eigenvalue}, we notice the fact $\zeta^{-1}=\omega^{2}$ (or $\zeta_h^{-1}=\omega_h^{2})$ and then it follows directly from \eqref{equ:evalerr}.
\end{proof}

\begin{table}[h!]
\centering
\caption{Material constants of aulminium/aluminium/epoxy}\label{tab:para}
\resizebox{\textwidth}{!}{
	\begin{tabular}{|l|l|l|l|}
\hline
Parameters & Aurum ($\Omega^-$) & Aluminium ($\Omega^-$)
 & Epoxy ($\Omega^+$)\\ \hline
Density $\rho$ ($kg/ m^3$)  & $19500$ & $2730$ & $1180$ \\
Lame's constant $\lambda$ ($N/ m^2$) & $4.23\times 10^{10}$ &
 $4.59\times 10^{10}$ &  $4.23\times 10^{9}$ \\
 Shear modulus $\mu$ ($N/ m^2$) & $2.99\times 10^{10}$ &
 $2.70\times 10^{10}$ &  $1.57\times 10^{9}$ \\ \hline
\end{tabular}}
\end{table}

\section{Numerical Examples}
\label{sec:example}
In this section, we shall use several benchmark numerical examples to
validate our theoretical results and illustrate the efficiency of the proposed unfitted numerical method in the computation of the band structure of the phononic crystal.  In the following tests, we shall consider the aurum/epoxy phononic crystal and the aluminium/epoxy phononic crystal as in \cite{WZLS2019}. The aurum (Au) scatters or the aluminium (Al) scatterers are embedded in the epoxy matrix.   Their material constants are documented in the \cref{tab:para}.
The transverse wave speed $c^s$ is  defined as
\begin{equation}\label{equ:speed}
	c^s = \sqrt{\mu^s/\rho^s}
\end{equation}
 for $s=\pm$. % The acoustic impedance ratio of the phononic crystal is $Z = \rho^-c^-/\rho^+c^+$.  %For the aurum/epoxy phononic crystal, we can calculate that $Z = 17.64$. Similarly, $Z = 6.64$ for aulminium/epoxy phononic crystal.
In all the following tests, the length of unit cell  $a$ is taken as 1.

To check the convergence rate for the unfitted Nitsche's method \eqref{equ:nit},  we shall  approximate the convergence rate of the exact error  by the rate of the  following the relative errors
\begin{equation*}
e_i =  \frac{\left|\omega_{i, h_j}^2 - \omega_{i, h_{j+1}}^2\right|}{ \omega_{i, h_j}^2},
\end{equation*}
where  $h_j$ is the mesh size of $j$th level meshes and $\omega_{i,h_j+1}^2$ is the $i$th eigenvalue on the $j$th level mesh.

\subsection{Square lattice with circular inclusion}
In the first numerical example,  consider the square lattice with circular inclusion as shown in \cref{fig:circle}. As was mentioned at the beginning of this section,  the inclusion scatter is either aurum or aluminium and the material constants are listed in \cref{tab:para}.   The radius of the circular material interface is $0.25$.

\begin{figure}[!h]
   \centering
   \subcaptionbox{\label{fig:circle_au}}
  {\includegraphics[width=0.47\textwidth]{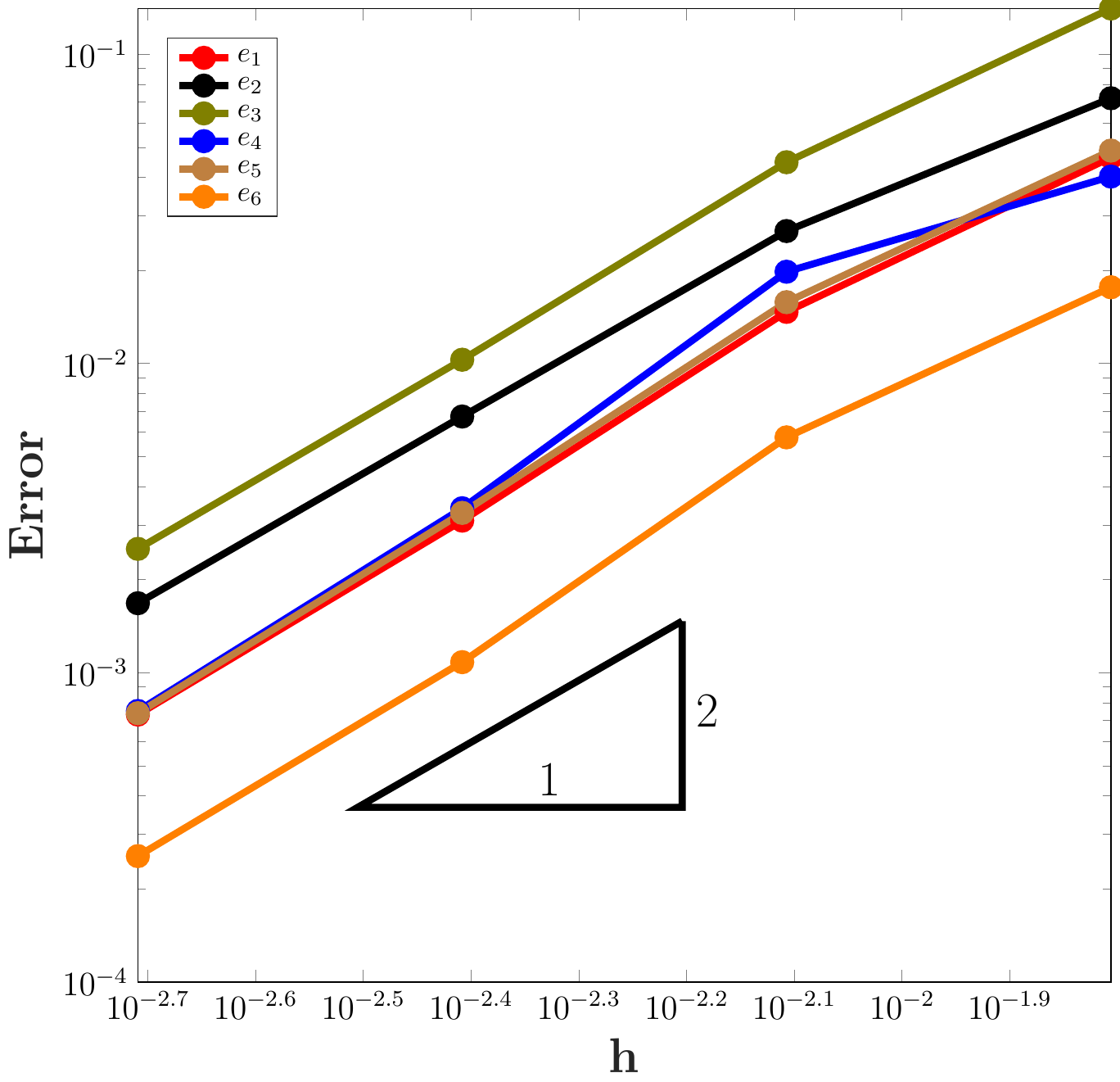}}
  \vspace{0.1in}
  \subcaptionbox{\label{fig:circle_al}}
   {\includegraphics[width=0.47\textwidth]{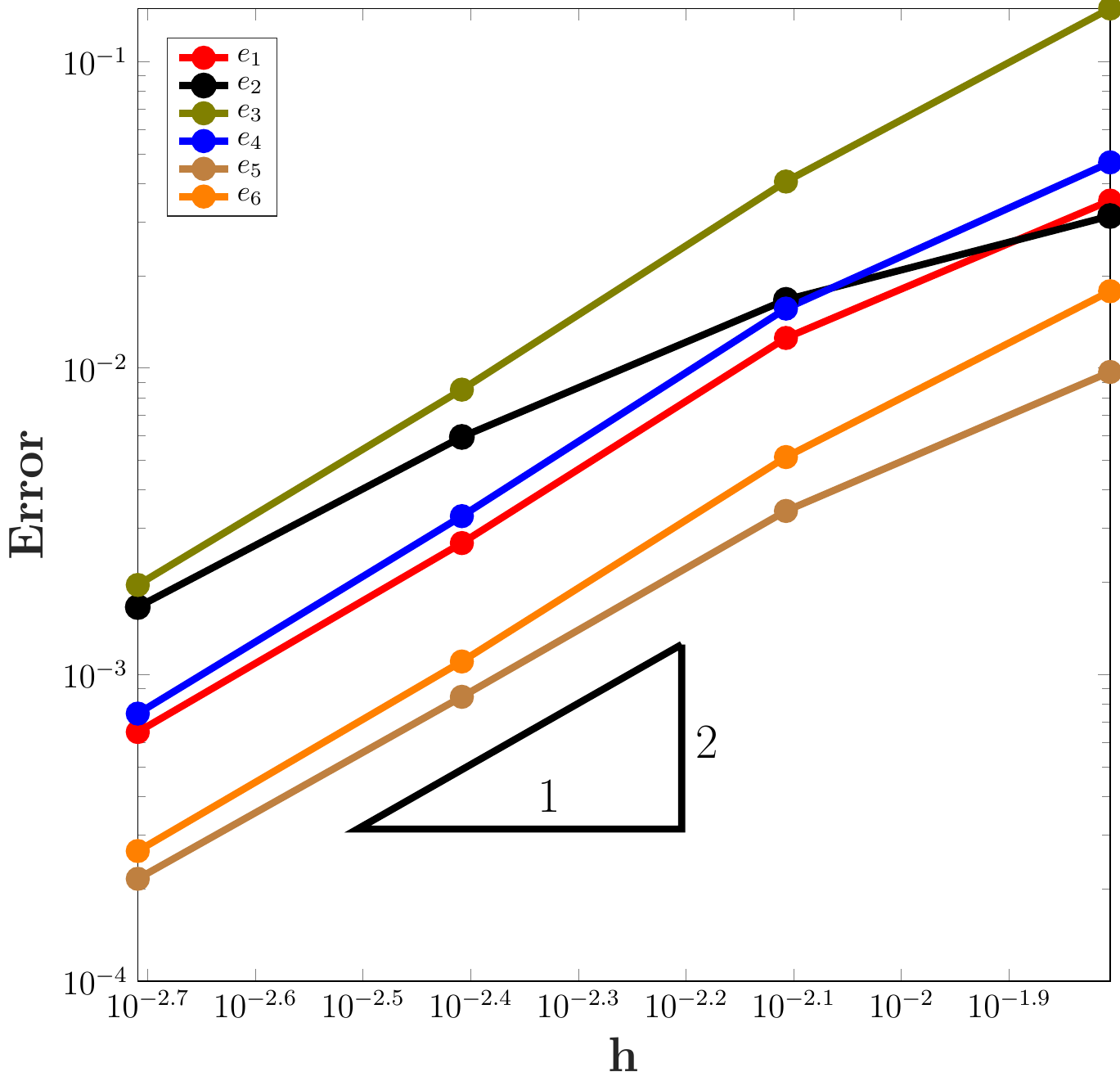}}
   \caption{Convergence rate of phononic crystal with circular   inclusion. (a): aurum/epoxy phononic crystal; (b): aulminium/epoxy phononic crystal.}
   \label{fig:circle_cr}
\end{figure}

\begin{figure}[!h]
   \centering
  \includegraphics[width=0.65\textwidth]{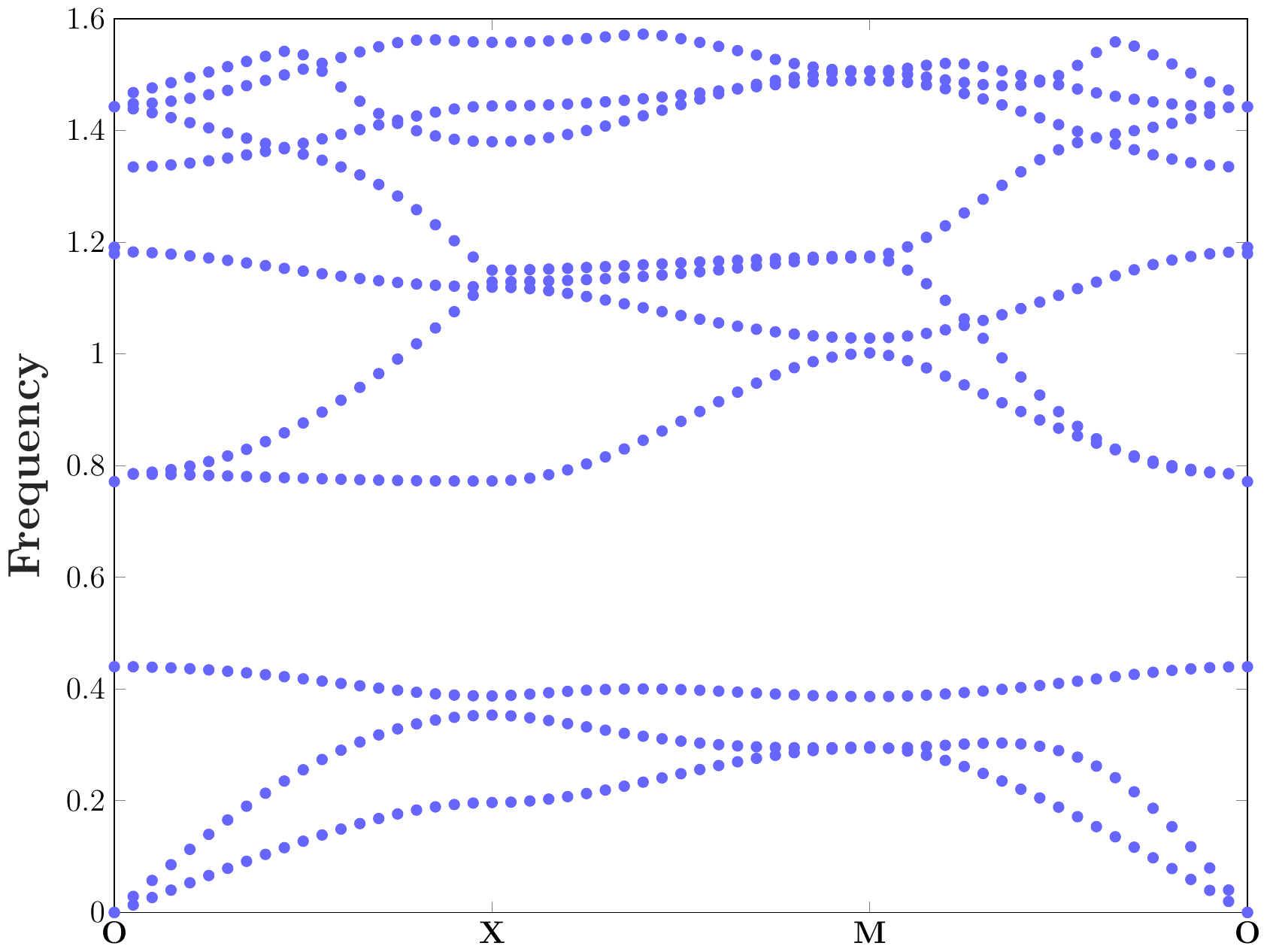}
   \caption{Band structure of aurum/epoxy phononic crystal with circular inclusion}
   \label{fig:bandcircle_1}
\end{figure}

Firstly, we test the convergence for unfitted Nitsche's method \eqref{equ:nit}.  We take  the quasi-momentum
$\bm{k} = (\pi, \pi)$.  The convergence history of the relative numerical errors is plotted in \cref{fig:circle_cr}.  Looking at \cref{fig:circle_cr}, it is apparent that the relative errors decay quadratically for both types of phononic crystals. The second-order convergence numerical results consist with the theoretical convergence rate predicted by \cref{thm:eigapp}.  Note the jump ratios of the material parameters are about 19  for the aurum/epoxy phononic crystal and about 17 for the alumina/epoxy phononic crystal. Despite the heterogeneous nature of the materials,  the proposed numerical method is theoretically and numerically proven to achieve the optimal convergence rate, which shows its potential in the efficient computation of the band structure.

\begin{figure}[!h]
   \centering
  \includegraphics[width=0.65\textwidth]{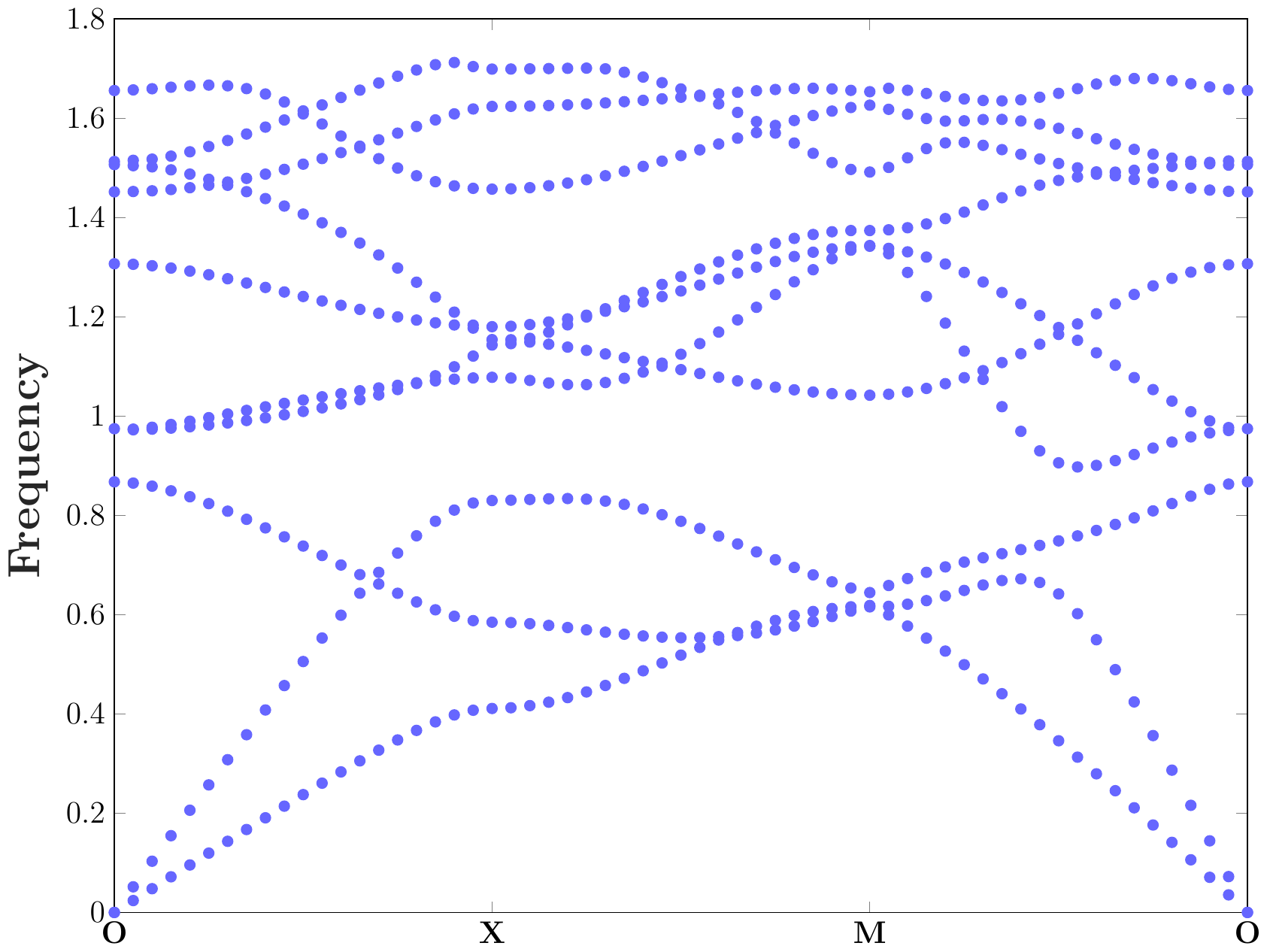}
   \caption{Band structure of aulumin/epoxy phononic crystal with circular inclusion}
   \label{fig:bandcircle_2}
\end{figure}

Now turn to the numerical computation of the band structure for the aurum/epoxy phononic crystal.  In the computation,  the mesh size is chosen to be $1/64$ and the quasi-momentum $k$ is taken on the boundary of the irreducible Brillouin zone.   In \cref{fig:bandcircle_1}, we plot the first ten normalized frequency along the direction O-X-M-O.
The normalized frequency is defined as $\omega a/(2\pi c^-)$ where $c^-$ is the wave speed of the scatters defined in \eqref{equ:speed}.
From the graph, we can see that there are  one small band-gap opens between the second eigencurve and the third eigencurve
and one relatively large band-gap opens between the third eigencurve and the fourth eigencurve.

Then, we focus on the computation of band structure of alumina/epoxy phononic crystal. The computational setup is the same as aurum/epoxy phononic crystal. The first ten normalized frequency is presented in \cref{fig:bandcircle_2}.   The most interesting aspect of this graph is that we only observe one relatively small band-gap between the third and fourth eigencurves. In contrast, we observed two band-gaps in the aurum/epoxy phononic crystal.

\begin{figure}[!h]
   \centering
   \subcaptionbox{\label{fig:flower_lattice}}
  {\includegraphics[width=0.3106\textwidth]{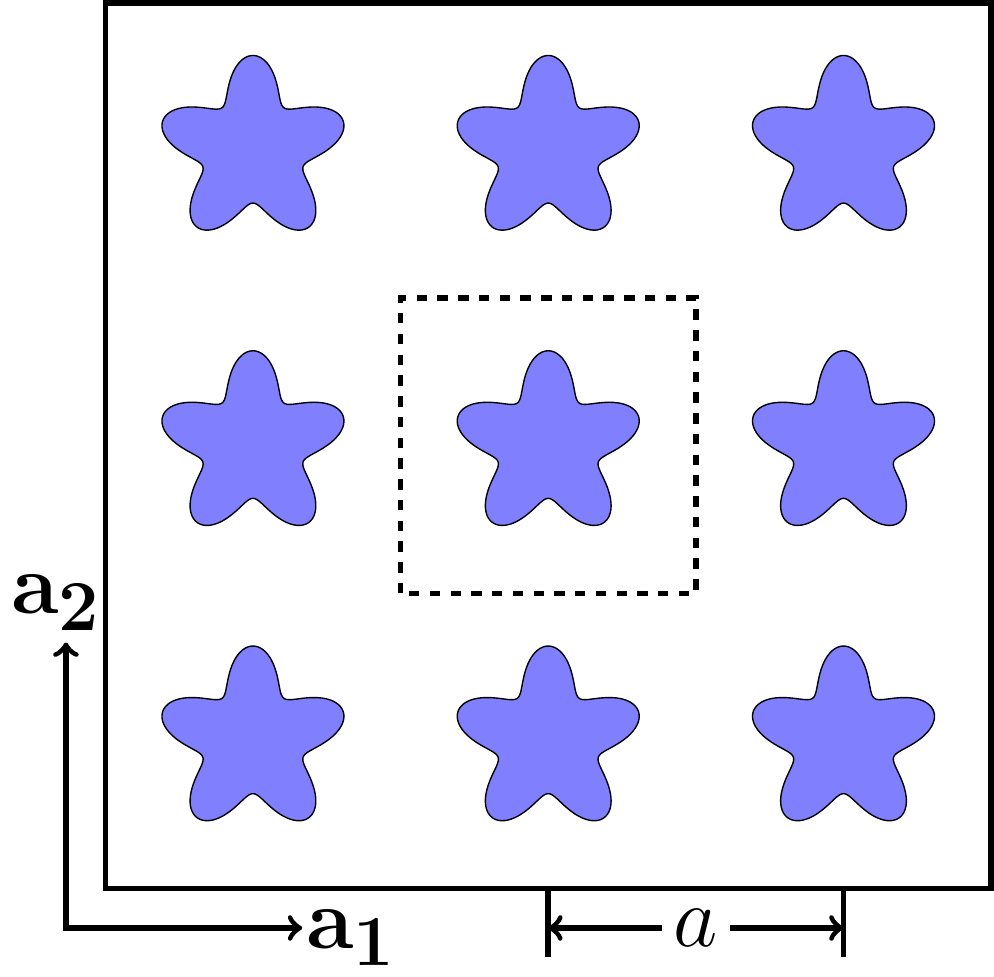}}
  \vspace{0.1in}
  \subcaptionbox{\label{fig:flow_unitcell}}
   {\includegraphics[width=0.275\textwidth]{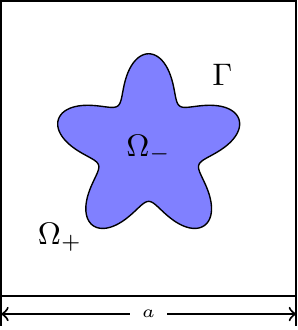}}
     \vspace{0.1in}
  \subcaptionbox{\label{fig:flower_brillouinzone}}
  {\includegraphics[width=0.306\textwidth]{fig/bz}}
   \caption{ Bravais lattice with flower shape inclusion.  (a): 2D square lattice; (b): the unit cell;
   (c): the First Brillouin Zone}
   \label{fig:flower}
\end{figure}

\subsection{Square lattice with flower shape inclusion}
Our second numerical example is  aurum/epoxy phononic crystal with flower shape inclusion. The phononic crystal is illustrated \cref{fig:flower}.  We conduct the computation in the fundamental cell $\Omega$ with length $a=1$, see \cref{fig:flow_unitcell}. The flower material interface curve in polar coordinate is given by
\begin{equation}
	r = \frac{1}{2} + \frac{\sin(5\theta)}{7},
\end{equation}
which contains both convex and concave parts.

Firstly, we verify the established theoretical results.  \cref{fig:flower_error} shows the convergence curve of the relative error for the first six eigenvalues. What stands out in the figure is that the relative error converges optimally at the rate of $\mathcal{O}(h^2)$ as predicted by \cref{thm:nopol}. The numerical results demonstrate the flexibility of the he proposed method in handling interfaces with complicate geometries.

\begin{figure}[!h]
   \centering
  \includegraphics[width=0.65\textwidth]{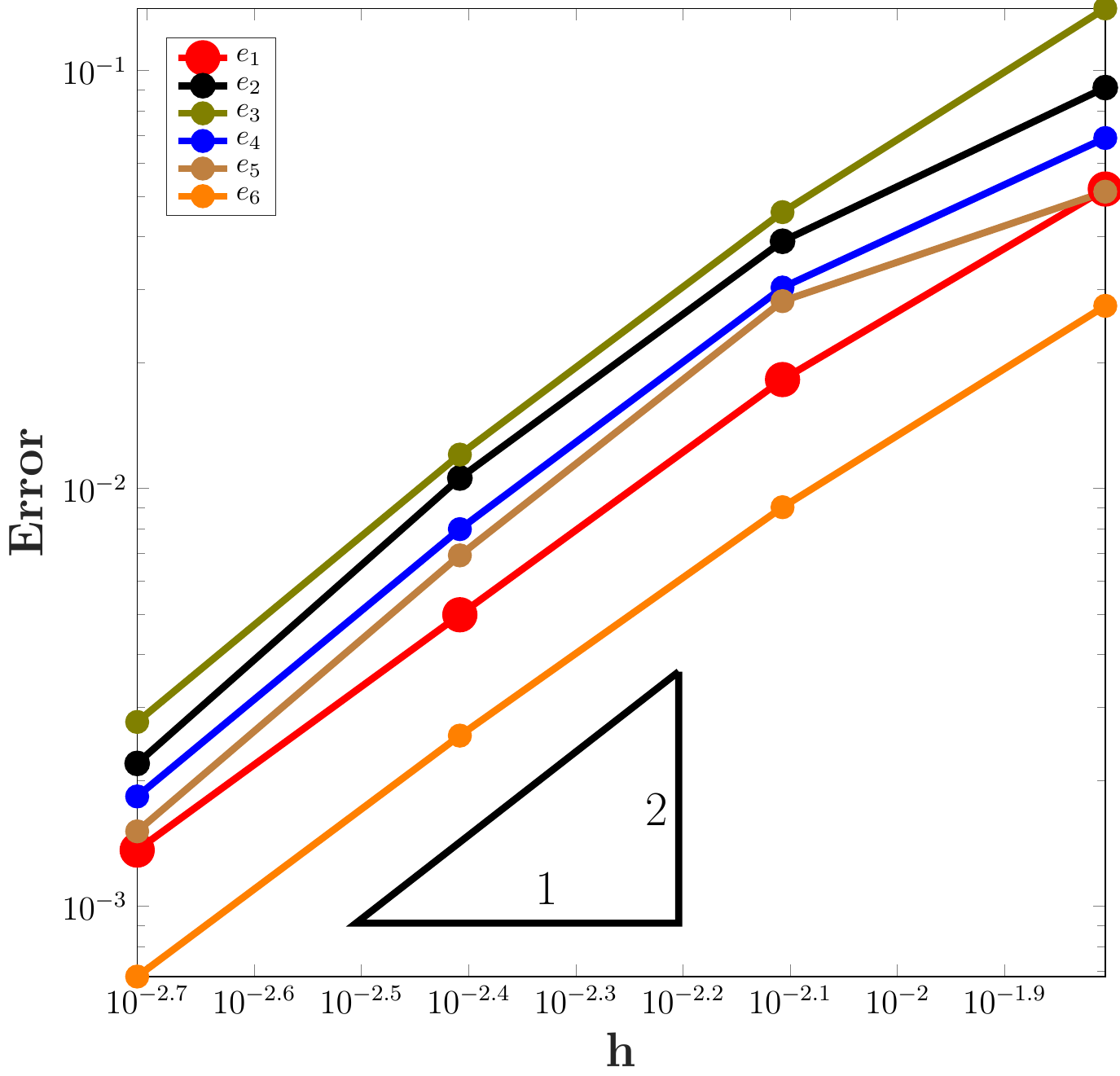}
   \caption{Convergence rate of aurum/epoxy phononic crystal with flower shape  inclusion }
   \label{fig:flower_error}
\end{figure}

\begin{figure}[!h]
   \centering
  \includegraphics[width=0.65\textwidth]{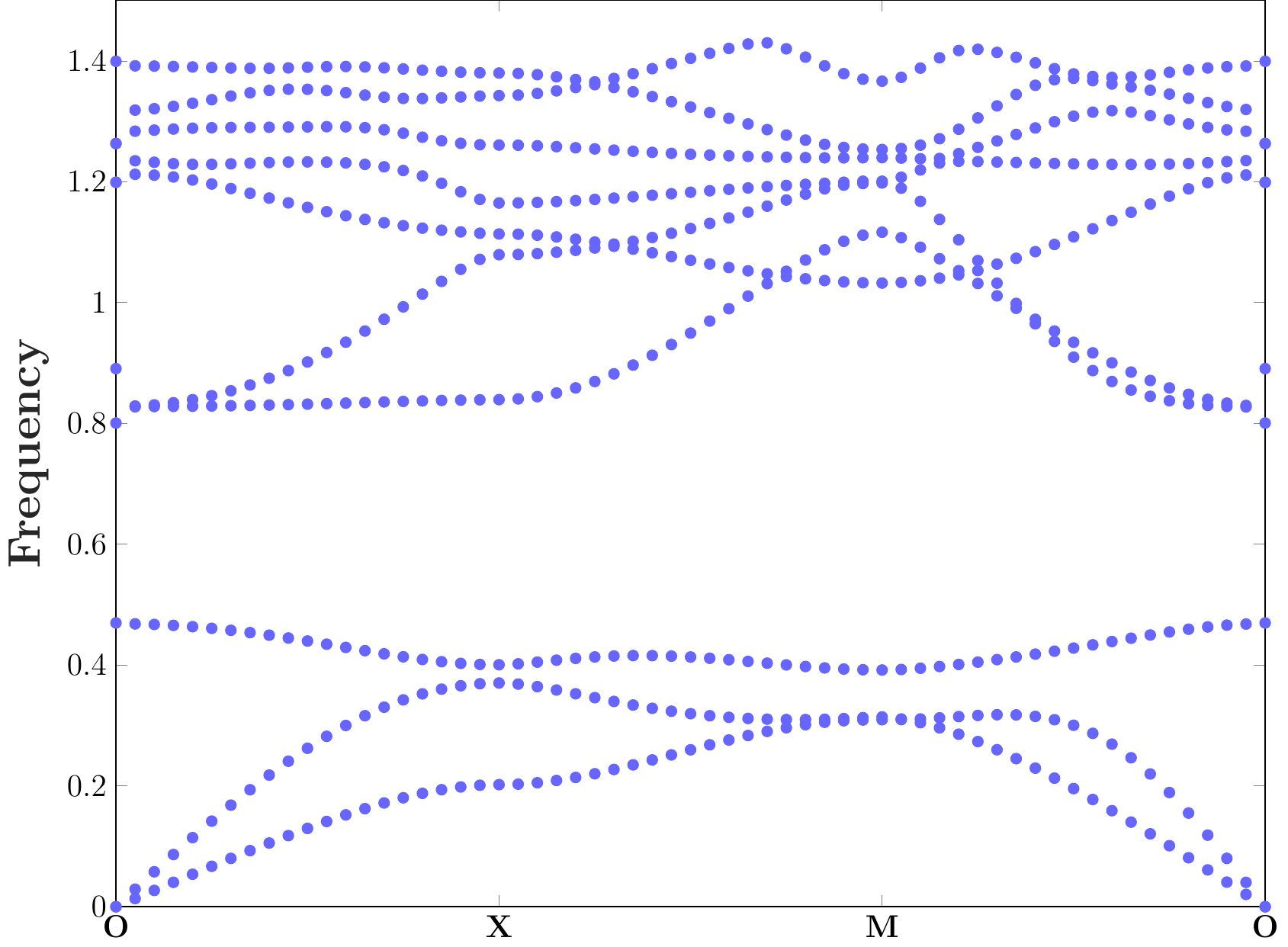}
   \caption{Band structure of aurum/epoxy phononic crystal with flower shape  inclusion}
   \label{fig:flower_band}
\end{figure}

Let us now turn to look the computation of the band structure. In this test, we take the mesh size $h = \frac{1}{64}$.  The first ten normalized frequency along O-M-X-O is plotted in \cref{fig:flower_band}.  Similar to aurum/epoxy phononic crystal with circular inclusion,  there are two band-gaps open: the relatively smaller band-gap is between the second and the third eigencurves; the relative larger band-gap is between the third and the fourth eigencurves.  An inspection of the data in \cref{fig:flower_band} reveals that the band-gap is relatively larger than the circular inclusion case.

\section{Conclusions}
\label{sec:conclusions}
In this paper,  a new finite element method for computing the band structure of phononic crystals with general material interfaces is proposed. To handle the quasi-periodic boundary condition, we transform the equation into an equivalent interface eigenvalue problem with periodic boundary conditions by applying the Floquet-Bloch transform.  The distinguishing feature of the proposed method is that it does not require the background mesh to fit the material interface which avoids the heavy burden of generating a body-fitted mesh and simplifies the impose of periodic boundary condition. Furthermore, the performance of the proposed method is theoretically founded. We show the well-posedness of the proposed method by using a delicate argument of the trace inequality.  With the aid of the Babu\v{s}ka-Osborn theory, we prove the proposed method achieves the optimal convergence result at the presence of material interfaces. The theoretical convergence rate is validated by two realistic numerical examples.  We also demonstrate the capability of the proposed methods in the computation of band structure without fitting the material interface.

%\appendix

\section*{Acknowledgment}
H.G. was partially supported by Andrew Sisson Fund of the University of Melbourne, X.Y. was partially supported by the NSF grant  DMS-1818592, and Y.Z. was partially supported by NSFC grant 11871299.

\bibliographystyle{siamplain}
\bibliography{references}

\begin{thebibliography}{10}

\bibitem{AdTr2020}
{\sc S.~C. Aduloju and T.~J. Truster}, {\em A primal formulation for imposing
  periodic boundary conditions on conforming and nonconforming meshes}, Comput.
  Methods Appl. Mech. Engrg., 359 (2020), pp.~112663, 29.

\bibitem{AFKRYZ2018}
{\sc H.~Ammari, B.~Fitzpatrick, H.~Kang, M.~Ruiz, S.~Yu, and H.~Zhang}, {\em
  Mathematical and computational methods in photonics and phononics}, vol.~235
  of Mathematical Surveys and Monographs, American Mathematical Society,
  Providence, RI, 2018.

\bibitem{AmKL2009}
{\sc H.~Ammari, H.~Kang, and H.~Lee}, {\em Asymptotic analysis of high-contrast
  phononic crystals and a criterion for the band-gap opening}, Arch. Ration.
  Mech. Anal., 193 (2009), pp.~679--714.

\bibitem{ALZ2019}
{\sc H.~Ammari, H.~Lee, and H.~Zhang}, {\em Bloch waves in bubbly crystal near
  the first band gap: a high-frequency homogenization approach}, SIAM J. Math.
  Anal., 51 (2019), pp.~45--59.

\bibitem{BaOs1989}
{\sc I.~Babu\v{s}ka and J.~E. Osborn}, {\em Finite element-{G}alerkin
  approximation of the eigenvalues and eigenvectors of selfadjoint problems},
  Math. Comp., 52 (1989), pp.~275--297.

\bibitem{BaOs1991}
{\sc I.~Babu\v{s}ka and J.~E. Osborn}, {\em Eigenvalue problems}, in Handbook
  of numerical analysis, {V}ol. {II}, Handb. Numer. Anal., II, North-Holland,
  Amsterdam, 1991, pp.~641--787.

\bibitem{BrennerScott2008}
{\sc S.~C. Brenner and L.~R. Scott}, {\em The mathematical theory of finite
  element methods}, vol.~15 of Texts in Applied Mathematics, Springer, New
  York, third~ed., 2008.

\bibitem{Burm2010}
{\sc E.~Burman}, {\em Ghost penalty}, C. R. Math. Acad. Sci. Paris, 348 (2010),
  pp.~1217--1220.

\bibitem{BCHLM2015}
{\sc E.~Burman, S.~Claus, P.~Hansbo, M.~G. Larson, and A.~Massing}, {\em
  Cut{FEM}: discretizing geometry and partial differential equations},
  Internat. J. Numer. Methods Engrg., 104 (2015), pp.~472--501.

\bibitem{CaHL2004}
{\sc Y.~Cao, Z.~Hou, and Y.~Liu}, {\em Finite difference time domain method for
  band-structure calculations of two-dimensional phononic crystals}, Solid
  State Communications, 132 (2004), pp.~539 -- 543.

\bibitem{CaRR2016}
{\sc F.~Casadei, J.~Rimoli, and M.~Ruzzene}, {\em Multiscale finite element
  analysis of wave propagation in periodic solids}, Finite Elements in Analysis
  and Design, 108 (2016), pp.~81 -- 95.

\bibitem{ChZo1998}
{\sc Z.~Chen and J.~Zou}, {\em Finite element methods and their convergence for
  elliptic and parabolic interface problems}, Numer. Math., 79 (1998),
  pp.~175--202.

\bibitem{Ciarlet2002}
{\sc P.~G. Ciarlet}, {\em The finite element method for elliptic problems},
  vol.~40 of Classics in Applied Mathematics, Society for Industrial and
  Applied Mathematics (SIAM), Philadelphia, PA, 2002.
\newblock Reprint of the 1978 original [North-Holland, Amsterdam; MR0520174 (58
  \#25001)].

\bibitem{CoMa2020}
{\sc C.~Comi and J.-J. Marigo}, {\em Homogenization {A}pproach and
  {B}loch-{F}loquet {T}heory for {B}and-{G}ap {P}rediction in 2{D} {L}ocally
  {R}esonant {M}etamaterials}, J. Elasticity, 139 (2020), pp.~61--90.

\bibitem{EcSi1994}
{\sc E.~N. Economou and M.~Sigalas}, {\em Stop bands for elastic waves in
  periodic composite materials}, The Journal of the Acoustical Society of
  America, 95 (1994), pp.~1734--1740.

\bibitem{Evans2008}
{\sc L.~C. Evans}, {\em Partial differential equations}, vol.~19 of Graduate
  Studies in Mathematics, American Mathematical Society, Providence, RI,
  second~ed., 2010.

\bibitem{GuYa2018}
{\sc H.~Guo and X.~Yang}, {\em Gradient recovery for elliptic interface
  problem: {III}. {N}itsche's method}, J. Comput. Phys., 356 (2018),
  pp.~46--63.

\bibitem{GuYZ2019}
{\sc H.~{Guo}, X.~{Yang}, and Y.~{Zhu}}, {\em {Unfitted Nitsche's method for
  computing wave modes in topological materials}}, arXiv e-prints,  (2019),
  arXiv:1908.06585, \url{https://arxiv.org/abs/1908.06585}.

\bibitem{HaHa2002}
{\sc A.~Hansbo and P.~Hansbo}, {\em An unfitted finite element method, based on
  {N}itsche's method, for elliptic interface problems}, Comput. Methods Appl.
  Mech. Engrg., 191 (2002), pp.~5537--5552.

\bibitem{HaHa2004}
{\sc A.~Hansbo and P.~Hansbo}, {\em A finite element method for the simulation
  of strong and weak discontinuities in solid mechanics}, Comput. Methods Appl.
  Mech. Engrg., 193 (2004), pp.~3523--3540.

\bibitem{HaLL2017}
{\sc P.~Hansbo, M.~G. Larson, and K.~Larsson}, {\em Cut finite element methods
  for linear elasticity problems}, in Geometrically unfitted finite element
  methods and applications, vol.~121 of Lect. Notes Comput. Sci. Eng.,
  Springer, Cham, 2017, pp.~25--63.

\bibitem{HuOs2020}
{\sc R.~Hu and C.~Oskay}, {\em Spectral variational multiscale model for
  transient dynamics of phononic crystals and acoustic metamaterials}, Comput.
  Methods Appl. Mech. Engrg., 359 (2020), pp.~112761, 26.

\bibitem{JuSt2009}
{\sc M.~Juntunen and R.~Stenberg}, {\em Nitsche's method for general boundary
  conditions}, Math. Comp., 78 (2009), pp.~1353--1374.

\bibitem{KaEc1999}
{\sc M.~Kafesaki and E.~N. Economou}, {\em Multiple-scattering theory for
  three-dimensional periodic acoustic composites}, Phys. Rev. B, 60 (1999),
  pp.~11993--12001.

\bibitem{Kittel2003}
{\sc C.~Kittel}, {\em {Introduction to Solid State Physics}}, John Wiley \&
  Sons, Inc., New York, 8th~ed., 2004.

\bibitem{KHDD1993}
{\sc M.~S. Kushwaha, P.~Halevi, L.~Dobrzynski, and B.~Djafari-Rouhani}, {\em
  Acoustic band structure of periodic elastic composites}, Phys. Rev. Lett., 71
  (1993), pp.~2022--2025.

\bibitem{LHWL2017}
{\sc E.~Li, Z.~C. He, G.~Wang, and G.~R. Liu}, {\em An ultra-accurate numerical
  method in the design of liquid phononic crystals with hard inclusion},
  Comput. Mech., 60 (2017), pp.~983--996.

\bibitem{LiCh2019}
{\sc W.~Li and W.~Chen}, {\em Simulation of the band structure for scalar waves
  in 2{D} phononic crystals by the singular boundary method}, Eng. Anal. Bound.
  Elem., 101 (2019), pp.~17--26.

\bibitem{ScZh1990}
{\sc L.~R. Scott and S.~Zhang}, {\em Finite element interpolation of nonsmooth
  functions satisfying boundary conditions}, Math. Comp., 54 (1990),
  pp.~483--493.

\bibitem{SiSo1995}
{\sc M.~M. Sigalas and C.~M. Soukoulis}, {\em Elastic-wave propagation through
  disordered and/or absorptive layered systems}, Phys. Rev. B, 51 (1995),
  pp.~2780--2789.

\bibitem{SiJe2003}
{\sc O.~Sigmund and J.~Jensen}, {\em Systematic design of phononic
  band\&\#x2013;gap materials and structures by topology optimization},
  Philosophical Transactions of the Royal Society of London. Series A:
  Mathematical, Physical and Engineering Sciences, 361 (2003), pp.~1001--1019.

\bibitem{StLuK2020}
{\sc S.~Sticko, G.~Ludvigsson, and G.~Kreiss}, {\em {High-order cut finite
  elements for the elastic wave equation}}, Advances in Computational
  Mathematics, 46 (2020), p.~45.

\bibitem{VGGZ2019}
{\sc C.~Valencia, J.~Gomez, and N.~Guar\'in-Zapata}, {\em A general-purpose
  element-based approach to compute dispersion relations in periodic materials
  with existing finite element codes}, Journal of Theoretical and Computational
  Acoustics, 27 (2019), p.~1950005.

\bibitem{VeBM2013}
{\sc I.~A. Veres, T.~Berer, and O.~Matsuda}, {\em Complex band structures of
  two dimensional phononic crystals: Analysis by the finite element method},
  Journal of Applied Physics, 114 (2013), p.~083519.

\bibitem{WZLS2019}
{\sc L.~Wang, H.~Zheng, X.~Lu, and L.~Shi}, {\em A {P}etrov-{G}alerkin finite
  element interface method for interface problems with {B}loch-periodic
  boundary conditions and its application in phononic crystals}, J. Comput.
  Phys., 393 (2019), pp.~117--138.

\bibitem{WuXi2019}
{\sc H.~Wu and Y.~Xiao}, {\em An unfitted {$hp$}-interface penalty finite
  element method for elliptic interface problems}, J. Comput. Math., 37 (2019),
  pp.~316--339.

\bibitem{Xu1982}
{\sc J.~Xu}, {\em Error estimates of the finite element method for the 2nd
  order elliptic equations with discontinuous coefficients}, J. Xiangtan Univ.,
  1 (1982), pp.~1--5.

\bibitem{ZZWS2016}
{\sc H.~Zheng, C.~Zhang, Y.~Wang, J.~Sladek, and V.~Sladek}, {\em Band
  structure computation of in-plane elastic waves in 2{D} phononic crystals by
  a meshfree local {RBF} collocation method}, Eng. Anal. Bound. Elem., 66
  (2016), pp.~77--90.

\end{thebibliography}
\end{document}